\documentclass[12pt,
psamsfonts]{amsart}

\usepackage{amssymb,amsfonts,amsmath}
\usepackage[all,arc]{xy}
\usepackage{enumerate}
\usepackage{mathrsfs}
\usepackage{fullpage}
\usepackage{xspace}
\usepackage[margin=1.0in]{geometry}
\usepackage{tcolorbox}
\usepackage{tikz-cd}
\usepackage{color}
\usepackage{aliascnt}
\usepackage[foot]{amsaddr}
\usepackage{hyperref}

\newtheorem{thm}{Theorem}[section]

\newaliascnt{theo}{thm}
\newtheorem{theo}[theo]{Theorem}
\aliascntresetthe{theo}

\newaliascnt{cor}{thm}
\newtheorem{cor}[cor]{Corollary}
\aliascntresetthe{cor}

\newaliascnt{prop}{thm}
\newtheorem{prop}[prop]{Proposition}
\aliascntresetthe{prop}

\newaliascnt{lem}{thm}
\newtheorem{lem}[lem]{Lemma}
\aliascntresetthe{lem}

\newaliascnt{conj}{thm}
\newtheorem{conj}[conj]{Conjecture}
\aliascntresetthe{conj}

\newaliascnt{que}{thm}
\newtheorem{que}[que]{Question}
\aliascntresetthe{que}

\newaliascnt{ass}{thm}
\newtheorem{ass}[ass]{Assumption}
\aliascntresetthe{ass}

\theoremstyle{remark}
\newaliascnt{rem}{thm}
\newtheorem{rem}[rem]{Remark}
\aliascntresetthe{rem}

\theoremstyle{definition}
\newtheorem{defn}[thm]{Definition}

\newtheorem{exmp}[thm]{Example}

\newtheorem{notn}[thm]{Notation}

\newtheorem{conv}[thm]{Convention}

\newcommand{\Z}{\mathbb{Z}\xspace}
\newcommand{\Q}{\mathbb{Q}\xspace}
\newcommand{\G}{\mathbb{G}\xspace}
\DeclareMathOperator{\Spec}{Spec}
\DeclareMathOperator{\res}{res}

\DeclareMathOperator{\ord}{ord}

\DeclareMathOperator{\dv}{div}

\DeclareMathOperator{\img}{Im}

\DeclareMathOperator{\Hom}{Hom}
\DeclareMathOperator{\Pic}{Pic}
\DeclareMathOperator{\Gal}{Gal}

\DeclareMathOperator{\Alb}{Alb}
\makeatletter
\let\c@equation\c@thm
\makeatother
\numberwithin{equation}{section}

\title{Zero-cycles on a product of elliptic curves over a $p$-adic field}

\author[*]{Evangelia Gazaki*} \address[*]{\normalfont Department of Mathematics, University of Virginia, Kerchof Hall, 141 Cabell Drive, Charlottesville, VA, 22904, USA. Email: \texttt{eg4va@virginia.edu}}
\author{Isabel Leal**}\address[**]{\normalfont Email: \texttt{isabel.leal.il@gmail.com}}
\begin{document}

\begin{abstract}
	We consider a product $X=E_1\times\cdots\times E_d$ of elliptic curves over a finite extension $K$ of $\Q_p$ with a combination of good or split multiplicative reduction. We assume that at most one of the elliptic curves has supersingular reduction. Under these assumptions, we prove that the Albanese kernel of $X$ is the direct sum of a finite group and a divisible group, extending work of Raskind and Spiess to cases that include supersingular phenomena. Our method involves studying the kernel of  the cycle map $CH_0(X)/p^n\rightarrow H^{2d}_{\text{\'{e}t}}(X, \mu_{p^n}^{\otimes d})$. We give specific criteria that guarantee this map is injective for every $n\geq 1$. When all curves have good ordinary reduction, we show that it suffices to extend to a specific finite extension $L$ of $K$ for these criteria to be satisfied. This extends previous work of Yamazaki and Hiranouchi. 
\end{abstract}

\maketitle

\section{Introduction} 
Let $X$ be a smooth, projective, and geometrically integral variety over a field $K$ having a $K$-rational point. We consider the group $CH_0(X)$ of zero-cycles on $X$ modulo rational equivalence and let $A_0(X)$ be the subgroup of zero-cycles of  degree zero. There is an abelian variety, $\Alb_X$, called the \textit{Albanese variety} of $X$, universal for maps from $X$ to abelian varieties, and an induced homomorphism, \[A_0(X)\rightarrow\Alb_X(K),\] called the \textit{Albanese map} of $X$. When $X$ is a curve, the group $A_0(X)$ coincides with $\Pic^0(X)$ and the Abel-Jacobi theorem tells us that the above map is an isomorphism. In higher dimensions, however,  the situation is far more mysterious and the Albanese map can have a very significant kernel, which we denote by $T(X)$. 

When $K$ is an algebraic number field, the fascinating Bloch-Beilinson conjectures predict that the Albanese kernel $T(X)$ is a torsion group. At the same time, the group $CH_0(X)$ is expected to be a finitely generated abelian group, so $T(X)$ must be  finite. On the other hand, the Albanese kernel is expected to be enormous for varieties with positive geometric genus over large fields like $\mathbb{C}$ or $\Q_p$. 

In this work, the case of interest to us is that of a $p$-adic base field $K$. In this case, the  expected structure of $T(X)$ is given by the following conjecture.
\begin{conj}\label{CTconj} Let  $X$ be a smooth projective and geometrically integral variety over a finite extension of $\Q_p$. The Albanese kernel $T(X)$ is the direct sum of a finite group and a divisible group.
\end{conj} A first version of this conjecture was formulated by Colliot-Th\'{e}l\`{e}ne (\cite{Colliot-Thelene1995}), and a later one by Raskind and Spiess (\cite{Raskind/Spiess2000}). In fact, Raskind and Spiess established the conjecture for a product $X=C_1\times\cdots\times C_d$ of smooth projective curves  all of whose Jacobians have a mixture of split multiplicative and good ordinary reduction.

More recently, S. Saito and K. Sato (\cite{Saito/Sato2010}) proved a weaker form of this conjecture. Namely, if $k$ is the residue field of the $p$-adic field $K$, they   established that, when $X$ has a regular projective flat model $\mathcal{X}$ over the ring of integers, $\mathcal{O}_K$, on which the reduced subscheme of the divisor $\mathcal{X}\otimes_{\mathcal{O}_K}k$ has simple normal crossings, the group $A_0(X)$ is the direct sum of a finite group and a group that is $m$-divisible for every integer $m$ coprime to the residue characteristic. The result has since been extended (\cite{CT2011}) to every smooth projective variety $X$ over a $p$-adic field. \autoref{CTconj} is still very open though and we do not have any general method to prove that the quotients $T(X)/p^n$ are ``small''. 

In this paper we focus on the case of a product $X=E_1\times\cdots\times E_d$ of elliptic curves over a $p$-adic field $K$. Following the method introduced by Raskind and Spiess in \cite{Raskind/Spiess2000}, 
we manage to  extend their result to include also supersingular reduction phenomena.  Our first result is a proof of \autoref{CTconj} in  the following case. 

\begin{theo}\label{maintheoremintro} Let $E_1,\cdots, E_d$ be elliptic curves over a $p$-adic field $K$ with either good or split multiplicative reduction. We assume that at most one of the curves has supersingular reduction. Then 
	the Albanese kernel, $T(X)$, of the product $X=E_1\times\cdots\times E_d$ is the direct sum of a finite group and a divisible group. 
\end{theo} 
\subsection{The cycle map}

An essential tool for the study of zero-cycles on varieties defined over arithmetic fields is the study of the cycle map to \'{e}tale cohomology, 
\[CH_0(X)/n\xrightarrow{c_n} H_{\text{\'{e}t}}^{2d}(X,\mu_{n}^{\otimes d}),\] where $d=\dim(X)$. The map $c_n$ is in general neither injective nor surjective. 
When $X$ is a smooth surface over a $p$-adic field  and $n$ is coprime to $p$, Esnault and Wittenberg (\cite{Esnault/Wittenberg}) give some far reaching computations of the kernel. However, when $n$ is a power of $p$, still very little is known. 

Our primary goal in this paper is to describe as much as possible 
the kernel of the cycle map $c_{p^n}$ in the context of \autoref{maintheoremintro} and for every $n\geq 1$. Unlike \autoref{maintheoremintro} that was already known when all the curves have good ordinary or split multiplicative reduction, the injectivity of $c_{p^n}$ was previously known only in very limited cases. Under the assumptions of \autoref{maintheoremintro}, it has been established by Raskind and Spiess (\cite{Raskind/Spiess2000}) and Hiranouchi (\cite{Hiranouchi2014}) that the map $c_{p^n}$ is injective under the additional assumption  that $E_i[p^n]\subset E_i(K)$ for every $i\in\{1,\cdots,d\}$. The only result independent of $n\geq 1$ is due to Yamazaki (\cite{Yamazaki2005}), who  proved injectivity of $c_n$ for every $n\geq 1$, for a product $X=C_1\times\cdots\times C_d$ of Mumford curves, that is, higher genus analogues of Tate curves.

In this article we focus on removing the strong $K$-rationality assumption, $E_i[p^n]\subset E_i(K)$, and  pass to the limit for $p^n$. For a product $X=E_1\times E_2$ of two elliptic curves not both having supersingular reduction,  we give sufficient criteria for the injectivity of $c_{p^n}$, for every $n\geq 1$.  These criteria depend heavily on the reduction type of $E_1, E_2$ (\autoref{ordord}, \autoref{ordss}, \autoref{Tate1}) and when they are satisfied, they give us very sharp results. Namely,  \autoref{maintheoremintro}  gives us a decomposition, $T(X)\simeq D\oplus F$, where $D$ is a divisible group and $F$ a finite group. Our method often allows us to fully compute the finite group $F$, which to our knowledge is the first result in this direction. 
\begin{exmp}\label{exampleintro} Let $X=E\times E$ be the self product of an elliptic curve over $K$ with good ordinary reduction. Under some mild assumptions, the cycle map $c_{p^n}$ is injective for every $n\geq 1$ and we have an isomorphism, $T(X)\simeq D\oplus\Z/p^n$, if $E[p^n]\subset E(K)$ for some $n\geq 1$ and $n$ is the largest with this property. If $n=0$, the Albanese kernel $T(X)$ is divisible. 
\end{exmp}

When the criteria for injectivity are not satisfied, we show that an obstruction to injectivity is very possible to exist (\autoref{kernelord}, \autoref{kernelss}). However, when all the curves have good ordinary reduction, we show that the obstruction goes away after extending to a tower of finite extensions of $K$. Namely, we prove the following theorem. 

\begin{theo}\label{mainintro2} Let $E_1,\cdots,E_d$ be elliptic curves over $K$ with good ordinary reduction. Let $X=E_1\times\cdots\times E_d$. Then there exists a finite extension $L$ of $K$ such that the cycle map \[CH_0(X\times_K L)/p^n\xrightarrow{c_{p^n}} H_{\text{\'{e}t}}^{2d}(X\otimes_K L,\mu_{p^n}^{\otimes d}),\] is injective for every $n\geq 1$. 
\end{theo}


\subsection{A corollary over global fields}
One special case when we get sharp results is when the elliptic curves have complex multiplication by an imaginary quadratic field. In this case we get the following global-to-local corollary for a product $X$ of elliptic curves defined over an algebraic number field. 
\begin{cor} Let $X=E\times E$ be the self-product of an elliptic curve over an algebraic number field $K$.  Assume that $E$ has complex multiplication by an imaginary quadratic field $M$. Let $X_v=X\otimes_K K_v$ be the base change to a completion of $K$ at a finite place $v$. Then the Albanese kernel,  $T(X_v)$, is divisible for almost all ordinary reduction places $v$ of $K$. 
\end{cor}

\subsection{Outline of our Method} In this paper we use a method introduced by Raskind and Spiess in \cite{Raskind/Spiess2000} and continued by more authors (\cite{Yamazaki2005}, \cite{Murre/Ramakrishnan2009}, \cite{hiranouchi/hirayama}, \cite{Hiranouchi2014}).
\subsection*{\texorpdfstring{Relation to the Somekawa $K$-group}{Relation to the Somekawa K-group}}  Raskind and Spiess reduced the study of the Albanese kernel $T(X)$ on a  product of curves to the study of the Somekawa $K$-group $K(K;A_1,\cdots,A_r)$ attached to abelian varieties $A_1,\cdots,A_r$ over $K$. 
This group is a generalization of the Milnor $K$-group, $K_r^M(K)$ of the field $K$. It is a quotient of the group 
\[\bigoplus_{L/K\text{ finite}}A_1(L)\otimes \cdots\otimes A_r(L)\]
first by a relation similar to the \textit{projection formula} of $CH_i(X)$ and then by a second relation coming from function fields of curves, known as \textit{Weil reciprocity}. 

The big advantage of this method is that  the group $K(K;A_1,\cdots,A_r)$ has specific gene-rators and relations. More importantly, when working over a $p$-adic field $K$, the projection formula is easy to use and in most cases it gives already enough relations that guarantee  that the quotients $T(X)/p^n$ are small. 

\subsection*{The Galois symbol} Similarly to the case of the Milnor $K$-groups, for an integer $n\geq 1$ invertible in $K$, there is a map to Galois cohomology, known as \textit{the generalized Galois symbol},
\[K(K;A_1,\cdots, A_r)/n\xrightarrow{s_{n}} H^r(K,A_1[n]\otimes\cdots\otimes A_r[n]).\] This map is constructed similarly to the Galois symbol of the Bloch-Kato conjecture, and it is conjectured by Somekawa (\cite{Somekawa1990}) to always be injective. Nonetheless, a counterexample has been found in \cite{Spiess/Yamazaki2009}, not for abelian coordinates but for the group $K(K;T,T)$ attached to two copies of a certain non-split torus over a two-dimensional local field $K$. 

Coming to the question of injectivity of the cycle map $c_{p^n}$ for products of curves over $p$-adic fields all having a $K$-rational point, this question has been reduced by Yamazaki to verifying the Somekawa conjecture for abelian  varieties. 

For elliptic curves $E_1,\cdots,E_r$ over a $p$-adic field $K$ satisfying the assumptions of \autoref{mainintro2}, the  conjecture has been established (\cite{Raskind/Spiess2000}, \cite{Hiranouchi2014}) under the  assumption that $E_i[p^n]\subset E_i(K)$, for $i=1,\cdots,r$. This is the assumption we would like to remove.  When $E_i[p]\subset E_i(K)$ for every $i=1,\cdots,d$, we introduce a new method to pass to the limit for $p^n$.  Roughly speaking, our method is based on the following principle. ``When $E_i[p]\subset E_i(K)$, the $K$-group $K(K;E_1,E_2)/p$ is generated by $p^N$ torsion points for sufficiently large $N\geq 1$''. When this is not achieved over $K$, we construct a tower, $K\subset L_1\subset\cdots\subset L_r$ of finite extensions so that this condition is achieved in the tower. 
When the curve $E_i$ has either good ordinary or split multiplicative reduction, we even manage to remove the assumption $E_i[p]\subset E_i(K)$ by using the theory of $p$-adic uniformization of elliptic curves.

We note that in all our computations we use a group larger than $K(K;E_1,\cdots,E_r)$. Namely, in the definition of the Somekawa $K$-group we forget the relations coming from function fields of curves.  For this larger group, we show that most of our conditions become necessary for injectivity. This, however, does not disprove the Somekawa conjecture.

\subsection*{Some Corollaries}   As a byproduct of our proofs, we obtain some important corollaries. First, in the context of \autoref{maintheoremintro}, we get a decomposition $T(X)\simeq F\oplus D$, with the finite group $F$  generated by $K$-rational points. We hope that this corollary could have potential applications over global fields. 

Moreover, using computations of Yamazaki (\cite{Yamazaki2005}), we obtain a corollary about the Brauer-Manin pairing, $CH_0(X)\times Br(X)\rightarrow\Q/\Z$, where by $Br(X)$ we denote the Brauer group of $X$ (\autoref{Brauer-Manin}).

We wish our methods could be used to establish \autoref{maintheoremintro} for any product of elliptic curves. Unfortunately, there is a very serious obstruction for a product $E_1\times E_2$ of two curves with supersingular reduction. Namely, in this case the easy projection formula of the Albanese kernel does not seem to give us enough relations that guarantee the quotient $T(E_1\times E_2)/p$ is finite. 
\subsection*{Notation} Unless otherwise specified, all cohomology groups considered in this paper will be over the \'{e}tale site. In particular, for a field $K$ we will denote by $H^i(K,-)$ the Galois cohomology groups of $K$. Moreover, we will denote the separable closure of a field $F$ by $\overline{F}$. 

If $L/K$ is an extension of fields and $X$ is a variety over $K$, we will denote the base change $X\otimes_K L$ by $X_L$. 



\vspace{2pt}
\subsection{Acknowledgements}
	We would like to express our sincere gratitude to Professors Bhargav Bhatt, Spencer Bloch, Jean-Louis Colliot-Th\'{e}l\`{e}ne,  Toshiro Hiranouchi, Kazuya Kato, Shuji Saito and 
	Takao Yamazaki for their interest in our work and their helpful comments and suggestions. We are particularly grateful to  T. Yamazaki for pointing out a mistake in an earlier version of this paper.  Finally, we would like to thank our referees for pointing out inaccuracies and providing many useful suggestions that helped improve significantly the paper. 
The first author is partially supported by the NSF grant DMS-2001605.

\bigskip
\section{\texorpdfstring{Mackey Functors and Somekawa $K$-groups}{Mackey Functors and Somekawa K-groups}} In this section we review the definition of the Somekawa  $K$-group $K(K;A_1,\cdots,A_r)$ for abelian varieties $A_1,\cdots,A_r$ over  a perfect field $K$. We start by reviewing the definition of a Mackey functor. 

Let $K$ be a perfect field. A Mackey functor $\mathcal{F}$ over $K$ is a presheaf on the category of \'{e}tale $K$-schemes having the following additional property. For every finite morphism $X\stackrel{f}{\longrightarrow} Y$ of \'{e}tale $K$-schemes, in addition to the restriction map $\mathcal{F}(Y)\stackrel{f^\star}{\longrightarrow}\mathcal{F}(X)$, there is also a push-forward map, $\mathcal{F}(X)\stackrel{f_\star}{\longrightarrow} \mathcal{F}(Y)$. The maps $f^\star$ and $f_\star$ satisfy certain functoriality conditions, for example for a composition $X\xrightarrow{f}Y\xrightarrow{g}Z$, we have an equality, $(f\circ g)_\star=f_\star\circ g_\star$. Moreover, there is a decomposition $\mathcal{F}(X_1\sqcup X_2)=\mathcal{F}(X_1)\oplus \mathcal{F}(X_2)$. Therefore, $\mathcal{F}$ is fully determined by its value  $\mathcal{F}(L):=\mathcal{F}(\Spec L)$ at every finite extension $L$ over $K$.  
 For a more detailed discussion on the properties of Mackey functors we refer to \cite[p. 13, 14]{Raskind/Spiess2000}. 
 
\subsection*{Notation}From now, if $K\stackrel{f}{\hookrightarrow}L$ is a finite extension of perfect fields, we will denote the restriction map by $\res_{L/K}:\mathcal{F}(K)\rightarrow\mathcal{F}(L)$ and the push-forward map by $N_{L/K}:\mathcal{F}(L)\rightarrow\mathcal{F}(K)$ and call it \textit{the norm}.

 \begin{exmp} Let $A$ be an abelian variety over $K$. Then $A$ induces a Mackey functor  by assigning to a finite extension $L/K$, $A(L):=\Hom(\Spec L, A)$. For a finite extension $F/L$,  the push-forward  is the norm map on abelian varieties, $N_{F/L}:A(F)\rightarrow A(L)$. 
\end{exmp}
Kahn proved in \cite{Kahn1992} that  the category $MF_K$ of Mackey functors on $(\Spec K)_{\text{\'{e}t}}$ is an \textit{abelian category with a tensor product $\otimes^M$}. For abelian varieties $A_1,\cdots, A_r$ over $K$, we review the definition of $A_1\otimes^M\cdots\otimes^M A_r$ below. The definition is in fact very similar for general Mackey functors $\mathcal{F}_1,\cdots,\mathcal{F}_r$, but here we only need the abelian variety case. 
\begin{defn} Let $A_1,\cdots, A_r$ be abelian varieties over a perfect field $K$. The Mackey product ${A_1\otimes^M\cdots\otimes^M A_r}$ is defined at a finite extension $L$ over $K$ as follows: 
\[(A_1\otimes^M\cdots\otimes^M A_r)(L):=\left(\bigoplus_{F/L\text{ finite}} A_1(F)\otimes\cdots\otimes A_r(F)\right)/R_1.\] Here $R_1$ is the subgroup generated by elements of the form 
\[\label{projectionformula} a_1\otimes\cdots\otimes N_{F'/F}(a_i)\otimes\cdots\otimes a_r-\res_{F'/F}(a_1)\otimes\cdots\otimes a_i\otimes\cdots\otimes \res_{F'/F}(a_r), \] 
where $F'\supset F\supset L$ is a tower of finite extensions of $K$, $a_i\in A_i(F')$ for some $i\in\{1,\cdots,r\}$, and $a_j\in A_j(F)$ for every $j\neq i$.  
\end{defn} 
\begin{notn} From now on we will be  using the standard symbol notation for the generators of $(A_1\otimes^M\cdots\otimes^M A_r)(L)$, namely $\{a_1,\cdots,a_r\}_{F/L}$ for $a_i\in A_i(F)$. 
\end{notn}

\subsection*{Norm and Restriction} Since $A_1\otimes^M\cdots\otimes^M A_r$ is a Mackey functor, there are norm and restriction maps corresponding to any finite extension $L/K$. Namely, 
\[\res_{L/K}:(A_1\otimes^M\cdots\otimes^M A_r)(K)\rightarrow (A_1\otimes^M\cdots\otimes^M A_r)(L),\]
and \[N_{L/K}:(A_1\otimes^M\cdots\otimes^M A_r)(L)\rightarrow (A_1\otimes^M\cdots\otimes^M A_r)(K).\] Moreover, we have the relation $N_{L/K}\circ\res_{L/K}=[L:K]$. 
\vspace{1pt}
\begin{rem}\label{normnotations} We note that the symbol $\{a_1,\cdots,a_r\}_{F/L}\in (A_1\otimes^M\cdots\otimes^M A_r)(L)$ is nothing but $N_{F/L}(\{a_1,\cdots,a_r\}_{F/F})$.  The defining relation $R_1$ is classically referred to as \textit{projection formula}. We rewrite it using the symbolic notation:
\begin{equation}\{a_1,\cdots, N_{F/L}(a_i),\cdots,a_r\}_{L/L}=N_{F/L}(\{\res_{F/L}(a_1),\cdots,a_i,\cdots,\res_{F/L}(a_r)\}_{F/F}). 
\end{equation}
\end{rem}
\vspace{1pt}
\subsection{\texorpdfstring{The Somekawa $K$-group}{The Somekawa K-group}} We are now ready to review the definition of the  $K$-group $K(K;A_1,\cdots,A_r)$ attached to abelian varieties over a perfect field $K$. 
\begin{defn} The Somekawa $K$-group $K(K;A_1,\cdots,A_r)$ is defined as 
\[K(K;A_1,\cdots,A_r)=(A_1\otimes^M\cdots\otimes^M A_r)(K)/R_2,\] where the subgroup $R_2$ is generated by the following family of elements. Let $C$ be a smooth complete curve over $K$ endowed with morphisms $g_i:C\rightarrow A_i$ for $i=1,\cdots,r$. Then for every function  $f\in k(C)^\times$ we require 
\[\sum_{x\in C}\ord_x(f)\{g_1(x),\cdots,g_r(x)\}_{\kappa(x)/K}\in R_2.\]
\end{defn} 
The above definition was given by Somekawa (\cite{Somekawa1990}), following a suggestion of K. Kato. Somekawa defined  more generally a $K$-group $K(K;G_1,\cdots, G_r)$ attached to semi-abelian varieties over a field $K$, that in the special case when $G_i=\G_m$ for every $i$, it turns out to be isomorphic to the Milnor $K$-group, $K_r^M(K)$. Recently this definition has been generalized to include more general coordinates. We refer to \cite{Ivorra/Ruelling2017} and \cite{KahnYamazaki2013} for more details. 
\begin{rem}
In most of this paper we will be using the Mackey product $(A_1\otimes^M\cdots\otimes^M A_r)(K)$ for our calculations. 
From now on we will use the same symbolic notation, $\{a_1,\cdots,a_r\}_{L/K}$,  for the generators of both the Mackey product and the Somekawa $K$-group. To avoid confusion, we will always clarify which group we are using. 
\end{rem}
\subsection*{Galois Symbol} Let $A_1,\cdots,A_r$ be abelian varieties over a perfect field $K$ and $n$ be an integer invertible in $K$. The Kummer maps, $A_i(L)/n\hookrightarrow H^1(L,A_i[n])$, together with the cup product and the norm map of Galois cohomology (i.e. the corestriction map) induce a generalized Galois symbol, 
\[s_n:K(K;A_1,\cdots,A_r)/n\rightarrow H^r(K,A_1[n]\otimes\cdots\otimes A_r[n]).\]
\begin{conj}\label{conjecture}(\cite{Somekawa1990}) The generalized Galois symbol $s_n$ is always injective. 
\end{conj} 
This conjecture is the analogue of the Bloch-Kato conjecture for the Somekawa $K$-groups. It is still very open in general and as already mentioned in the introduction, a counterexample has been found (\cite{Spiess/Yamazaki2009}) for non-abelian coordinates. 

\subsection{Relation to zero-cycles} For a  product $X=E_1\times\cdots\times E_d$ \footnote{We note that the construction of Raskind and Spiess was a lot more general, for a product of smooth complete and geometrically connected curves all having a $K$-rational point, and the filtration was constructed using the Somekawa $K$-groups attached to their Jacobian varieties. }  of elliptic curves over a field $K$, Raskind and Spiess (\cite[Corollary 2.4.1]{Raskind/Spiess2000}) constructed a finite decreasing filtration $F^0\supset F^1\supset \cdots\supset F^N\supset 0$ of $CH_0(X)$ such that the successive quotients $F^i/F^{i+1}$ are isomorphic to Somekawa $K$-groups of the form $K(K;E_{i_1},\cdots,E_{i_r})$, for $1\leq i_1<\cdots<i_r\leq d$.  More precisely, they proved an isomorphism 
\[CH_0(X)\simeq\Z\oplus\bigoplus_{1\leq \nu\leq d}\bigoplus_{1\leq i_1<\cdots<i_\nu\leq d}K(K;E_{i_1},\ldots, E_{i_\nu})\]
 Additionally, the subgroups $F^1$ and $F^2$ coincide with $A_0(X)$ and $T(X)$ respectively (\cite[Remark 2.4.2 (b)]{Raskind/Spiess2000}, \cite[Example 2.2]{Yamazaki2005}). 

Yamazaki then showed (\cite[Proposition 2.4]{Yamazaki2005}) that in the above set-up, the injectivity of the cycle map \[CH_0(X)/n\xrightarrow{c_n} H^{2d}_{\text{\'{e}t}}(X,\mu_n^{\otimes d})\] can be reduced to verifying the Somekawa conjecture for all the Galois symbols
\[K(E_{i_1},\cdots,E_{i_r})/n\xrightarrow{s_n} H^{2r}(K;E_{i_1}[n]\otimes\cdots\otimes E_{i_r}[n]).\]

\subsection{\texorpdfstring{Injectivity in the $p$-adic case}{Injectivity in the p-adic case}}
 From now on we focus on the case of a $p$-adic field $K$. In fact we make the following convention. 
\begin{conv} From now on, unless specified otherwise, we assume that $K$ is a $p$-adic field with ring of integers $\mathcal{O}_K$, maximal ideal $\mathfrak{m}_K$ and residue field $k$.  Moreover, we assume that all the elliptic curves considered in this paper have \textit{split semistable reduction}. 
\end{conv} 
We give an overview of the status of \autoref{conjecture}. 
\subsection*{\texorpdfstring{When $n$ is coprime to $p$}{When n is coprime to p}} In this case the problem is easier to handle. When at least two of the abelian varieties $A_1,\cdots, A_r$ have good reduction, the injectivity of $s_n$ follows from the following stronger result. 
\begin{theo} (Raskind and Spiess, \cite[Theorem 3.5]{Raskind/Spiess2000}) If $n$ is coprime to $p$ and at least two of the abelian varieties $A_1,\cdots, A_r$ have good reduction, the group $K(K;A_1,\cdots,A_r)$ is $n$-divisible. In particular the Galois symbol vanishes, $s_n=0$.
\end{theo}
\subsection*{\texorpdfstring{When $n$ is a power of $p$}{When n is a power of p}} Proving injectivity of $s_{p^n}$ for $n\geq 1$ over a $p$-adic field $K$ is a mixed characteristic problem of great difficulty. Raskind and Spiess  described a general method that could be used to establish injectivity, under the additional assumption that $\mu_{p^n}\subset K$ and $A_i[p^n]\subset A_i(K)$, for $i=1,\cdots, r$. We briefly review this method only for two elliptic curves $E_1,E_2$ over $K$, to keep the notation simple. 

The main idea is to relate the generalized Galois symbol $s_{p^n}$ to the classical Galois symbol of the Bloch-Kato conjecture, 
\[K_2^M(K)/p^n\stackrel{g_{p^n}}{\longrightarrow} H^2(K,\mu_{p^n}^{\otimes 2}).\] When $\mu_{p^n}\subset K$, the latter has a concrete description in terms of central simple algebras. Moreover, for a symbol $\{x,y\}$ the following equivalence is known. 
\begin{tcolorbox}\label{MilnorCriterion}
\begin{align}& g_{p^n}(\{x,y\})=0 \Leftrightarrow\; x\in N_{
K(\sqrt[p^n]{y})/K}(K(\sqrt[p^n]{y})^\times)
\Leftrightarrow \;y\in N_{
K(\sqrt[p^n]{x})/K}(K(\sqrt[p^n]{x})^\times).\nonumber
\end{align}
\end{tcolorbox}
The steps of the method are as follows:
\begin{itemize}
\item Because of the $K$-rationality assumption, we can fix an isomorphism $E_i[p^n]\simeq(\mu_{p^n})^{\oplus 2}$, for $i=1,2$. This in turn gives us isomorphisms
\[H^2(K,E_1[p^n]\otimes E_2[p^n])\simeq\bigoplus^4 H^2(K,\mu_{p^n})\simeq \bigoplus^4 Br(K)[p^n]\simeq\bigoplus^4 \Z/p^n.\]
\item The next step is to describe realizations of the Mackey functors $E_1/p^n$ and $E_2/p^n$ as subfunctors of $\G_m/p^n$. If such realizations exist and are compatible with the Kummer map, 
$E_i(L)/p^n\hookrightarrow H^1(L,E_i[p^n])$, for every finite extension $L$ of $K$, then this description can be used in order to  compute the image of 
\[(E_1/p^n\otimes^M E_2/p^n)(K)\stackrel{s_{p^n}}{\longrightarrow}H^2(K,E_1[p^n]\otimes E_2[p^n]), \] using known facts about the classical Galois symbol. 
\item After computing the image, one could try to show an isomorphism, 
\[
	\begin{tikzcd} 
	(E_1/p^n\otimes^M E_2/p^n)(K)\ar{r}{\simeq}\ar{d} & \img(s_{p^n})\\
	K(K;E_1,E_2)/p^n\ar{ur}{s_{p^n}}.
	\end{tikzcd}
	\] This would imply that the projection $(E_1/p^n\otimes^M E_2/p^n)(K)\twoheadrightarrow K(K;E_1,E_2)/p^n$ is an equality and in particular the diagonal $s_{p^n}$ is an isomorphism as well. 
\end{itemize}
\medskip

Following the above method, Raskind and Spiess established injectivity of $s_{p^n}$  for abelian varieties with a mixture of good ordinary and split multiplicative reduction under the above $K$-rationality assumption. Their work has since been generalized by Yamazaki (\cite{Yamazaki2005}) who managed to remove the  assumption but only for abelian varieties with split multiplicative reduction. More recently Hiranouchi (\cite{Hiranouchi2014}) extended the original computation to include also supersingular reduction elliptic curves. We will review his result in the next section. See also \cite{Murre/Ramakrishnan2009} for an alternative proof for the self product $E\times E$ of an ordinary reduction elliptic curve. 

\begin{rem} We note that for elliptic curves $E_1,\cdots, E_r$ over the $p$-adic field $K$ with $r\geq 3$, proving \autoref{conjecture} amounts to showing the $K$-group $K(K;E_1,\cdots, E_r)$ is divisible. As we will see later in the paper \autoref{morecurves},  this usually follows as an easy corollary after proving the conjecture for the product of two elliptic curves.  
\end{rem}

\begin{rem}\label{MackeyRemark} Raskind and Spiess imagined that the Mackey functor relation should be enough to establish injectivity when working with abelian varieties over $p$-adic fields, while the function field relation of $K(K;A_1,\cdots, A_r)$ is expected to be crucial for varieties over number fields. We will show in the forthcoming sections, however, that without the $K$-rationality assumption, injectivity is not always guaranteed by the Mackey functor relations, even over $p$-adic fields. 
\end{rem}
\vspace{1pt}
\subsection{\texorpdfstring{Decomposing the Mackey functor $E/p$ for an elliptic curve $E$}{Decomposing the Mackey functor E/p for an elliptic curve E}}  Let $E$ be an elliptic curve over $K$ such that $E[p]\subset E(K)$. In this subsection we review the aforementioned realization of the Mackey functor $E/p$ as a subfunctor of $\G_m/p$. We first need some information about the filtration of $K^\times$ arising from the groups of units, $\mathcal{O}_K^\times\supset U^1_K\supset U^2_K\supset\cdots$, where $U^i_K=1+\mathfrak{m}_K^i$. 
\subsection*{The Unit groups as Mackey Functors} We assume $\mu_p\subset K$. We define the following filtration of the group $K^\times/p:=K^\times/(K^{\times})^p$. For $i\geq 0$, 
\[\overline{U}^i_K:=\img(U^i_K\rightarrow K^\times/p).\] 
Note that the assumption $\mu_p\subset K$ implies that $p-1$ divides the absolute ramification index $e_K$. We denote $\displaystyle e_0(K):=\frac{e_K}{p-1}$. 
The graded quotients $\overline{U}^i_K/\overline{U}_K^{i+1}$ are known to satisfy the following. 
\begin{lem}\label{units}(\cite[Lemma 2.1.4]{kawachi2002}) Assume $\mu_p\subset K$. 
\begin{enumerate}[(a)]
\item If $0\leq i< pe_0(K)$ and $i$ is coprime to $p$, then $\overline{U}^i_K/\overline{U}_K^{i+1}\simeq
k$.
\item If $0\leq i< pe_0(K)$ and $i$ is divisible by $p$, then 
$\overline{U}^i_K/\overline{U}_K^{i+1}\simeq
1$.
\item If $i=pe_0(K)$, then $\overline{U}^i_K/\overline{U}_K^{i+1}\simeq
\Z/p$.
\item If $i>pe_0(K)$, then $\overline{U}_K^i=1$. 
\end{enumerate}
\end{lem}
\begin{defn} For every $i\geq 0$, we define a Mackey functor, $\overline{U}^i$ as follows. If $L$ is a finite extension of $K$, then 
\[\overline{U}^i(L):=\overline{U}_L^{ie(L/K)}.\] For a finite extension $F/L$, the norm $N_{F/L}$ and restriction maps $\res_{F/L}$ are induced by the norm and restriction on $\G_m$. 
\end{defn}
\begin{theo}\label{decomposition}(\cite{kawachi2002}, \cite{takemoto}) Let $E$ be an elliptic curve over $K$ with split semistable reduction such that $E[p]\subset E(K)$. The Mackey funtor $E/p$ is calculated as follows,
\[E/p=\left\{
        \begin{array}{lll}
      \G_m/p, & \text{if E is a Tate curve} \\
     \overline{U}^0\oplus\overline{U}^{pe_0(K)}, & \text{if E has ordinary reduction}\\
     \overline{U}^{pt}\oplus\overline{U}^{p(e_0(K)-t)}, & \text{if E has supersingular reduction.}
        \end{array}
      \right.
\] 
\end{theo} For the case of a supersingular reduction elliptic curve, there is an invariant $t$ that appears in the above decomposition. This invariant is defined as follows. 
\begin{defn} For an elliptic curve $E$ over $K$ with supersingular reduction such that $E[p]\subset E(K)$, the invariant $t$  is defined to be
\[t:=\max\{i\geq 0:P\in\hat{E}(\mathfrak{m}^i_K), \;\text{for every }P\in E[p]\}.\] 
\end{defn} The fact that $E$ has supersingular reduction yields that $t\geq 1$. Moreover, $t<pe_0(K)$. This is because for every $j\geq pe_0(K)$ the group $\hat{E}(\mathfrak{m}_K^j)$ is known to be torsion free. (\cite{Silverman2009})

\begin{rem} We note that the decomposition given in \autoref{decomposition} is constructed using the image of the Kummer map, $E(L)/p\hookrightarrow H^1(L,E[p])$, for $L$ a finite extension of $K$. In fact, the assumption $E[p]\subset E(L)$ for every such extension $L$ gives an isomorphism between  $H^1(L,E[p])$ and  $L^\times/p\oplus L^\times/p$, so via the Kummer map we may view $E(L)/p$ as subgroup of $L^\times/p\oplus L^\times/p$. 
This compatibility allows us to use the above decomposition for the generalized Galois symbol, $(E_1\otimes^M E_2)/p(K)\rightarrow H^2(K,E_1[p]\otimes E_2[p])$. 
\end{rem} 
\begin{exmp}\label{generalizedtoclassical} Assume for example that $E_1$ is an elliptic curve with good ordinary reduction and $E_2$ is a Tate curve, with $E_i[p]\subset E_i(K)$ for $i=1,2$. Let $a\in E_1(K)$ and $b\in E_2(K)$ be two closed points. Under the decomposition given by \autoref{decomposition}, the image of $b\in E_2(K)/p$ can be thought of as the class of a point $b\in K^\times/p$. Moreover, the image of $a\in E_1(K)/p$ is of the form $(a_1,a_2)$ with $a_1\in\overline{U}^0_K$ and $a_2\in\overline{U}_K^{pe_0(K)}$. Then, 
\[s_p(\{a,b\}_{K/K})=(g_p(\{a_1,b\}),g_p(\{a_2,b\}))\in\Z/p\oplus\Z/p,\] where $g_p:K_2^M(K)/p\rightarrow Br(K)[p]\simeq\Z/p$ is the classical Galois symbol. 
\end{exmp}

\vspace{2pt}

\section{The main theorems} We make the following assumption for the rest of this section. 
\begin{ass}\label{assumption}  Unless otherwise specified, \textit{$E_1,E_2$ shall denote elliptic curves over the $p$-adic field $K$, both with split semistable reduction, and such that at least one of them does not have supersingular reduction}. Moreover, if $E$ is an elliptic curve over $K$ with good reduction, we will denote by $\mathcal{E}$ its N\'{e}ron model, (which is an abelian scheme over $\Spec(\mathcal{O}_K)$) and by $\overline{E}:=\mathcal{E}\otimes_{\mathcal{O}_K}k$ the special fiber (which is an elliptic curve over the residue field $k$). 
\end{ass} 
 We consider the local Galois symbol 
\[s_{p^n}:K(K;E_1,E_2)/p^n\rightarrow H^2(K,E_1[p^n]\otimes E_2[p^n]).\]

We recall the following result. 
\begin{theo}\label{Hiranouchiinjec} (Hiranouchi, \cite{Hiranouchi2014}) If $E_i[p^n]\subset E_i(K)$ for $i=1,2$, the map $s_{p^n}$ is injective. 
\end{theo}

Hiranouchi used an involved argument to prove injectivity of $s_p$. The injectivity of $s_{p^n}$ follows by diagram chasing and induction. 
\autoref{Hiranouchiinjec}, together with the computation of the image of $s_p$ by Hiranouchi and Hirayama (\cite[Theorem 3.4]{hiranouchi/hirayama}), yield the following theorem. 
\medskip
\begin{theo}\label{image1}(Hiranouchi, \cite{Hiranouchi2014}, Hiranouchi-Hirayama, \cite{hiranouchi/hirayama}) Assume that $E_i[p^n]\subset E_i(K)$ for $i=1,2$. Then 
\[\frac{K(K;E_1,E_2)}{p^n}\simeq \frac{E_1\otimes^M E_2}{p^n}(K)\simeq
\begin{cases}
	\Z/p^n, & \text{if }E_1, E_2\text{ have the same reduction type} \\  \Z/p^n\oplus\Z/p^n, & \text{if }E_1, E_2\text{ have different reduction type}
	\end{cases}\]
\end{theo} 
\begin{rem}\label{generators} It is important to note that, while proving injectivity of $s_p$, Hiranouchi showed that the group $(E_1\otimes^M E_2)/p$ can be generated by symbols of the form $\{a,b\}_{K/K}$. 
\end{rem}
\medskip

Our first goal is to remove the strong  assumption $E_i[p^n]\subset E_i(K)$ and pass to the limit for $p^n$. 
 Starting with the case when $E_i[p]\subset E_i(K)$, for $i=1,2$, the following quite general lemma provides a sufficient criterion that guarantees injectivity of $s_{p^n}$ for every $n\geq 1$. 
\begin{lem}\label{torsionpointscriterion} Let $E_1,E_2$ be elliptic curves over $K$. Assume that $E_i[p^n]\subset E_i(K)$, $i=1,2$, for some $n\geq 1$ which is the largest with this property. Further, assume that the Galois symbol $s_p$ is injective.  If $K(K;E_1,E_2)/p$ can be generated by symbols of the form $\{a,b\}_{K/K}$ with either $a\in E_1[p^n]$ or $b\in E_2[p^n]$, then the group $p^n K(K;E_1,E_2)$ is $p$-divisible, that is $p^n K(K;E_1,E_2)=p^s K(K;E_1,E_2)$, for every $s>n$. In particular, the Galois symbol $s_{p^m}$ is injective for every $m\geq 1$. 
\end{lem}
\begin{proof} 
Let $x\in p^n K(K;E_1,E_2)$. We may write $x=p^ny$ for some $y\in K(K;E_1,E_2)$.
We consider the image of $y$ in $K(K;E_1,E_2)/p$. By the assumption of the lemma, we may write $y$ in the following form,
\begin{eqnarray}\label{relation} y=\sum_i\{a_i,b_i\}_{K/K}+pz,\end{eqnarray}
where $z\in K(K;E_1,E_2)$ and either $a_i\in E_1[p^n]$ or $b_i\in E_2[p^n]$. We conclude that  the element $p^n y$ is $p$-divisible. 



Notice that this implies that the Galois symbol $s_{p^m}$ is injective for every $m\geq 1$. For, if $m\leq n$, the injectivity follows from \autoref{Hiranouchiinjec}. On the other hand, if $x\in\ker(s_{p^{s}})$ for some $s>n$, a simple induction and diagram chasing shows that $x\in p^n K(K;E_1,E_2)$ and the claim follows by the $p$-divisibility of $p^n K(K;E_1,E_2)$. 

\end{proof} 
\begin{rem}\label{vanishing}
It is clear that if the elliptic curves $E_1,E_2$ satisfy \autoref{torsionpointscriterion}, then \autoref{maintheoremintro} holds for the product $X=E_1\times E_2$. Indeed, if $D$ is the maximal $p$-divisible subgroup of $K(K;E_1,E_2)$, then we can write $K(K;E_1,E_2)\simeq D\oplus F$ for some subgroup $F$. The lemma together with the fact that $K(K;E_1,E_2)/p^i$ is a finite group for every $i\geq 1$ imply that subgroup $F$ is finite. In many cases, we won't be able to verify the assumptions of  \autoref{torsionpointscriterion}. However, very often we will be able to show the weaker condition that the group $p^N K(K;E_1,E_2)$ is $p$-divisible for some $N\geq n$, by showing that the $K$-group $K(K;E_1,E_2)/p$ can be generated by symbols of the form $\{a,b\}_{K/K}$, with either $a\in E_1[p^N](K)$ or $b\in E_2[p^N](K)$, for some $N>n$. 
\end{rem} 

\begin{rem} We note that if either \autoref{torsionpointscriterion} or the weaker condition of \autoref{vanishing} holds, we can show that the finite summand $F$ of $K(K;E_1,E_2)$ can be generated  by symbols  $\{x,y\}_{K/K}$ defined over $K$, as long as this is true for the group $K(K;E_1,E_2)/p$. This follows  inductively using the exact sequence 
$K(K;E_1,E_2)/p\stackrel{p^{m}}{\longrightarrow} K(K;E_1,E_2)/p^{m+1}\longrightarrow K(K;E_1,E_2)/p^{m}\longrightarrow 0. $ When $E_1,E_2$ satisfy \autoref{assumption}, this has been proved by Hiranouchi (see \autoref{generators}). 
\end{rem}
\vspace{1pt}

\subsection{The product of two elliptic curves with ordinary reduction}  
Our first computation will be for the product $E_1\times E_2$ of two elliptic curves over $K$, both having good ordinary reduction. We start with a preliminary discussion, which includes some background on elliptic curves of such reduction type.
\subsection*{The connected-\'{e}tale exact sequence} Let $E$ be an elliptic curve over $K$ with good ordinary reduction, and $n\geq 1$ a positive integer. 
The $G_K$-module $E[p^{n}]$ has a one-dimensional $G_K$-invariant submodule. Namely, we have a short exact sequence of $G_K$-modules, 
\begin{eqnarray}\label{ses5}&&0\rightarrow 
		E[p^{n}]^\circ\rightarrow E[p^{n}]\rightarrow E[p^{n}]^{et}\rightarrow 0,\end{eqnarray}
where $E[p^{n}]^\circ:=\hat{E}[p^{n}]$ are the $p^{n}$-torsion points of the formal group $\hat{E}$ of $E$. 

If we further assume that $E[p^{n-1}]\subset E(K)$, then after a finite unramified extension $L_0/K$ of degree coprime to $p$, this sequence becomes \begin{eqnarray}\label{ses0}&&0\rightarrow 
		\mu_{p^{n}}\rightarrow E[p^{n}]\rightarrow \Z/p^{n}\rightarrow 0.\end{eqnarray}
\label{connectedetale} The  short exact sequence \eqref{ses5} is known as the \textit{connected-\'{e}tale exact sequence} for $E[p^{n}]$. The reason for the name is that this exact sequence can be obtained from the exact sequence of \textit{finite flat group schemes} over $\Spec(\mathcal{O}_K)$, 
\[0\rightarrow\mathcal{E}[p^{n}]^\circ\rightarrow\mathcal{E}[p^{n}]\rightarrow \mathcal{E}[p^{n}]^{et}\rightarrow 0,\] by extending to the generic fiber. Here we denoted by $\mathcal{E}$ the N\'{e}ron model of $E$. 
\medskip
\subsection*{The Serre-Tate parameter}\label{SerreTate} We next assume that $\mu_{p}\subset K$ and that we have  a  non-splitting  short exact sequence of finite flat group schemes over $\Spec(\mathcal{O}_K)$, 
\begin{eqnarray}\label{ses4}0\rightarrow 
		\mu_{p}\rightarrow \mathcal{E}[p]\rightarrow \Z/p\rightarrow 0.\end{eqnarray}  This in particular means that $\hat{E}[p]\subset \hat{E}(\mathcal{O}_K)$. 
		Notice that $\mathcal{E}[p]$ defines in this case a non-trivial element of $\mathcal{E}xt^1_{\mathcal{O}_K}(\Z/p,\mu_p)\simeq H^1_{fppf}(\mathcal{O}_K,\mu_p)$. This group is isomorphic to $\mathcal{O}^\times_K/\mathcal{O}_K^{\times p}$ and therefore the extension $\mathcal{E}[p]$ (or equivalently the Galois module $E[p]$) corresponds to a unit $u\in\mathcal{O}_K^\times$ that is not a $p$th power. That is, the sequence \eqref{ses4} becomes split after extending to the finite extension $K(\sqrt[p]{u})$. The unit $u$ is known as the Serre-Tate parameter of $E$. For more information we refer to \cite[Chapter 8, Section 9]{KatzMazur}. 
		
		Next we want to give a new interpretation of this unit $u$ that will be more helpful for our purposes. We first need some information about the Mackey functor $\hat{E}/[p]$, where $\hat{E}$ is the formal group of $E$ and $[p]:\hat{E}\rightarrow\hat{E}$ is the multiplication by $p$ isogeny. Because we assumed that $E$ has ordinary reduction, the isogeny $[p]$ has height one. Recall that $\hat{E}$ induces a Mackey functor which is defined at a finite extension $L/K$ as $\hat{E}(L):=\hat{E}(\mathcal{O}_L)$ with the obvious norm and restriction maps.

		\begin{prop}\label{formalgroup}  Let $E$ be an elliptic curve over $K$ with good ordinary reduction, and $\hat{E}$ be its formal group. Assume that $\hat{E}[p]\subset\hat{E}(\mathcal{O}_K)$ and $\mu_p\subset K$. Then we have an isomorphism of Mackey functors, $\hat{E}/[p]\simeq \overline{U}^1\simeq \overline{U}^0.$
		\end{prop} 
		\begin{proof} The first isomorphism follows directly from \cite[Theorem 2.1.6, Corollary 2.1.7]{kawachi2002}, if we apply it to the height 1 isogeny, $[p]:\hat{E}\rightarrow\hat{E}$. To make this more precise, for every finite extension $L/K$ we have an isomorphism, $\hat{E}(\mathcal{O}_L)/[p]\hat{E}(\mathcal{O}_L)\simeq \overline{U}_L^{p(e_0(L)-t(L))+1}$, where the invariant $t(L)$ is defined as, \[t(L)=\max\{i\geq 0: P\in \hat{E}(\mathfrak{m}^i_L), \text{ for every }P\in\hat{E}[p]\}.\] 
We claim that $t(L)=e_0(L)$. Since for every finite extension $L'/L$ we have equalities, $t(L')=e(L'/L)t(L)$ and $e_0(L')=e(L'/L)e_0(L)$, it suffices to prove this equality after extending to $L(E[p])$. But then the result follows from \autoref{decomposition}. The second isomorphism follows from \autoref{units}. 
		
		\end{proof}
		
	 We next consider the short exact sequence of abelian groups, \begin{eqnarray}\label{ses1}0\rightarrow \hat{E}(\mathcal{O}_K)\stackrel{j}{\longrightarrow}E(K)\stackrel{r}{\longrightarrow}\overline{E}(k)\rightarrow 0,\end{eqnarray} where $E(K)\stackrel{r}{\longrightarrow} \overline{E}(k)$ is the reduction map.  By tensoring \eqref{ses1} with $\Z/p$ and using \autoref{formalgroup} we get an exact sequence,
\[\overline{U}^0_K\stackrel{j}{\longrightarrow} E(K)/p\stackrel{r}{\longrightarrow} \overline{E}(k)/p\rightarrow 0.\] The claim is that the map $\overline{U}_K^0\stackrel{j}{\longrightarrow}E(K)/p$ is not injective. Namely there is a unit $u\in \overline{U}_K^0$ that generates the kernel and this unit is the Serre-Tate parameter of $E$. To construct $u$, we proceed as follows. Let $b\in\overline{E}[p](k)$ be a $p$-torsion point  with $b\neq 0$. Such a point exists because $\overline{E}[p]\simeq E^{et}[p]\simeq\Z/p$. Since the reduction map is surjective, we may choose a lift $\tilde{b}$ of $b$ in $E(K)$. We claim that $r(p\tilde{b})=pb=0$, but $p\tilde{b}$ is nonzero. Indeed, if $p\tilde{b}=0$, then  $\tilde{b}$ would be a $K$-rational $p$-torsion point of $E$, which would contradict the non-splitting of the short exact sequence \eqref{ses4}.  Next,  the exactness of the sequence \eqref{ses1} yields the existence of a unit $u\in\mathcal{O}_K^\times$ such that $j(u)=p\tilde{b}$. The class of $u\in\mathcal{O}_K^\times/\mathcal{O}_K^{\times p}$ is independent of the choice of lift. To finish the claim, we need to verify that $u\not\in K^{\times p}$. Assume to the contrary that $u$ is a $p$th power, i.e., $u=v^p$ for some $v$ in $K^\times$. Then the equation $pj(v)=p\tilde{b}$ yields that $\tilde{b}-j(v)$ is a non-zero $p$-torsion point of $E$. Since $r(\tilde{b})=b\neq 0$, this would imply that $E[p]\subset E(K)$, which is a contradiction. 
\vspace{2pt} 
\subsection*{Injectivity in the wild case} In this subsection we consider the question of injectivity of the Galois symbol for two  elliptic curves $E_1,E_2$ with ordinary reduction. We will often work with the Mackey product, $(E_1\otimes^M E_2)(K)/p^n$ instead of the Somekawa $K$-group $K(K;E_1,E_2)/p^n$. To distinguish between the two groups, we will call the map  \[(E_1\otimes^M E_2)(K)/p^n\stackrel{s_{p^n}}{\longrightarrow} H^2(K,E_1[p^n]\otimes E_2[p^n])\] the \textit{Mackey functor Galois symbol}. Recall that the latter has the same image as the actual Galois symbol 
\[K(K;E_1,E_2)/p^n\stackrel{s_{p^n}}{\longrightarrow} H^2(K,E_1[p^n]\otimes E_2[p^n]),\] but it might have a larger kernel.
\medskip
\begin{theo}\label{ordord} Assume $\mu_p\subset K$. Let $E_1,E_2$ be elliptic curves over $K$ with good ordinary reduction and let $n\geq 0$ be the largest integer such that $E_i[p^n]\subset E_i(K)$, for $i=1,2$. Assume:  
\begin{itemize}
\item The extension $L=K(E_1[p^{n+1}],E_2[p^{n+1}])$ has wild ramification.
\item For $i=1,2$ we have short exact sequences of $G_K$-modules, 
\[0\rightarrow \mu_{p^{n+1}}\rightarrow E_i[p^{n+1}]\rightarrow \Z/p^{n+1}\rightarrow 0.\]
\end{itemize}
	 Then the Galois symbol $s_{p^{m}}:K(K;E_1,E_2)/p^{m}\rightarrow H^2(K,E_1[p^{m}]\otimes E_2[p^{m}])$ is injective for every $m\geq 1$. In particular, if $n=0$, the $K$-group $K(K;E_1,E_2)$ is $p$-divisible. 
\end{theo}
\begin{proof} 
 We first prove injectivity when $n\geq 1$, which implies $E_i[p]\subset E_i(K)$ for $i=1,2$. 
 Without loss of generality assume that $K(E_1[p^{n+1}])/K$ is wildly ramified. Then there exists a $p^n$-torsion point $w\in E_1[p^n]$ such that the extension $L_w=K(\frac{1}{p}w)$ is wildly ramified over $K$. We will show that the assumption of \autoref{torsionpointscriterion} holds, more precisely, that $ K(K;E_1,E_2)/p$ is generated by symbols of the form  $\{w,y\}_{K/K}$ with $y\in E_2(K)$. 

By \autoref{Hiranouchiinjec} we get that the Galois symbol $s_p$ is injective and can be computed by the following composition,   
		\[K(K;E_1,E_2)/p\simeq(\overline{U}^0\otimes^{M}\overline{U}^{0})(K)\stackrel{g_p}{\longrightarrow} Br(K)[p]\simeq\Z/p,\] where $g_p$ is the classical Galois symbol.  Moreover,  recall \eqref{decomposition} that we have a decomposition, $E_1(K)/p\simeq\overline{U}_K^0\oplus\overline{U}_K^{pe_0(K)}$. We consider the image of $w=(w_1,w_2)$ under this decomposition. Since $L_w/K$ is wildly ramified, we necessarily have $w_1\neq 0$ and even stronger, that $w_1\in\overline{U}_K^i\setminus\overline{U}_K^{i+1}$ for some $i$ coprime to $p$.  
	To prove the claim, it suffices therefore to show that there exists some $y\in \overline{U}^0(K)\subset E_2(K)/p$ such that $g_p(\{w_1,y\})\neq 0$.  Equivalently, it suffices to show that there exists a unit $y\in\overline{U}^0(K)$ such that $y\not\in N_{K(\sqrt[p]{w})/K}(K(\sqrt[p]{w})^\times)$. The existence of such a $y$ follows by \cite[p. 86, Corollary 7]{serre1979local}. 
	 \medskip
	 
 Now we prove injectivity when $n=0$. In this case, either $E_1[p]\not\subset E_1(K)$ or $E_2[p]\not\subset E_2(K)$.  We will show that $(E_1\otimes^M E_2)/p=0$, which in particular implies that the $K$-group $K(K;E_1,E_2)$ is $p$-divisible. 

We have that $\mu_p\subset K$, and for $i=1,2$, there are short exact sequences of $G_K$ modules 
\[0\rightarrow \mu_p\rightarrow E_i[p]\rightarrow \Z/p\rightarrow 0.\]
Without loss of generality, assume that $K(E_1[p])/K$ is wildly ramified. In particular,  the extension $0\rightarrow  \mu_p\rightarrow E_1[p]\rightarrow \Z/p\rightarrow 0$ does not split, and the corresponding Serre-Tate parameter $u$ has the property that $K(\sqrt[p]{u})/K$ is a totally ramified degree $p$ extension.

Recall that the category of Mackey functors is abelian with a tensor product. The short exact sequence of abelian groups \eqref{ses1} induces a short sequence of Mackey functors, 
\[0\rightarrow \hat{E}_i\rightarrow E_i\rightarrow [E_i/\hat{E}_i]\rightarrow 0,\] where  $[E_i/\hat{E}_i]$ is the Mackey functor defined as follows. For a finite extension $F/K$, denote the  residue field of $F$ by $k_F$, and let $[E_i/\hat{E}_i](F):=\overline{E}_i(k_F)$. Moreover, the restriction $\res_{F/K}:[E_i/\hat{E}_i](K)\rightarrow [E_i/\hat{E}_i](F)$ is the usual restriction, $\overline{E}_i(k)\stackrel{\res_{F/K}}{\longrightarrow} \overline{E}_i(k_F)$, while the norm $N_{F/K}:[E_i/\hat{E}_i](F)\rightarrow [E_i/\hat{E}_i](K)$ is the map $e(F/K)\cdot N_{F/K}:\overline{E}_i(k_F)\rightarrow \overline{E}_i(k)$. The fact that $[E_i/\hat{E}_i]$ is a Mackey functor has been shown by Raskind and Spiess (\cite[p. 15]{Raskind/Spiess2000}). We consider the sequence for $i=2$ and we apply the right exact functor, $\otimes\Z/p$. Using \autoref{formalgroup}, we obtain an exact sequence of Mackey functors, 
\[\overline{U}^0\otimes^M\overline{U}^0\stackrel{j}{\longrightarrow}\overline{U}^0\otimes^M E_2/p\stackrel{r}{\longrightarrow}\overline{U}^0\otimes^M[E_2/\hat{E}_2]/p\rightarrow 0.\] 
We claim that $\overline{U}^0\otimes^M[E_2/\hat{E}_2]/p=0$. Indeed, consider a symbol  $\{x,y\}_{F/K}$, where $F$ is some finite extension of $K$, $x\in\overline{U}^0(F)$, and $y\in[E_2/\hat{E}_2](F)$. There is  $y'\in[E_2/\hat{E}_2](\overline{F})$ such that $py'=y$; more precisely, for some finite unramified extension $F'/F$, we can find $y'\in[E_2/\hat{E}_2](F')$ such that $py'=y$. But, since $F'/F$ is unramified, the norm map $N_{F'/F}:\overline{U}^0(F')\rightarrow\overline{U}^0(F)$ is surjective (\cite[p. 81, Proposition 1]{serre1979local}), so the claim follows. 
We conclude that there is an exact sequence $\overline{U}^0\otimes^M\overline{U}^0\stackrel{j}{\longrightarrow}\overline{U}^0\otimes^M E_2/p\rightarrow 0$. 

Using a similar argument, we obtain an exact sequence of Mackey functors, 
\[\overline{U}^0\otimes^M E_2/p\stackrel{j}{\longrightarrow}E_1/p\otimes^M E_2/p\stackrel{r}{\longrightarrow}
[E_1/\hat{E}_1]/p\otimes^M E_2/p\rightarrow 0.\]  We can again conclude that $[E_1/\hat{E}_1]/p\otimes^M E_2/p=0$, because the elliptic curve $E_2$ has good reduction, and hence for every finite unramified extension $F'/F$, the norm map $N_{F'/F}:E_2(F')\rightarrow E_2(F)$ is surjective (\cite[Corollary 4.4]{Mazur1972}). Finally, the two exact sequences induce a surjection, 
\[\overline{U}^0\otimes^M\overline{U}^0\stackrel{j}{\longrightarrow} E_1/p\otimes^M E_2/p\rightarrow 0.\] 
Evaluating at $\Spec(K)$, we get a surjection, 
\[(\overline{U}^0\otimes^M\overline{U}^0)(K)\stackrel{j}{\longrightarrow} (E_1/p\otimes^M E_2/p)(K)\rightarrow 0.\] 
The group $(\overline{U}^0\otimes^M\overline{U}^0)(K)$ is isomorphic via the classical Galois symbol to $\Z/p$ (\cite[Lemma 3.3]{Hiranouchi2014}). It suffices therefore to show that some non-zero element of $(\overline{U}^0\otimes^M\overline{U}^0)(K)$ is mapped to zero under $j$. But this now is easy, using the description of the Serre-Tate parameter $u$ described in the beginning of this section. Namely, by our assumption, the extension $K(\sqrt[p]{u})$ is totally ramified of degree $p$ over $K$, and hence there exists a unit $b\in\overline{U}^0(K)$ such that $g_p(\{u,b\}_{K/K})\neq 0$. Thus we get a generator $\{u,b\}_{K/K}$ of 
$(\overline{U}^0\otimes^M\overline{U}^0)(K)$ which is clearly mapped to zero under $j$.

\end{proof}

\vspace{1pt}
\subsection*{Possible Kernel in the unramified Case} We will now show that if the assumption of \autoref{ordord} does not hold, the Mackey functor Galois symbol has a nontrivial kernel. We emphasize that this does not disprove  \autoref{conjecture}, since the group $(E_1/p\otimes^M E_2/p)(K)$ could in general be larger than the Somekawa $K$-group $K(K;E_1,E_2)/p$.
\begin{prop}\label{kernelord} Let $E_1, E_2$ be elliptic curves over $K$ with good ordinary reduction. We assume that $\mu_p\subset K$ and the $G_K$ modules $E_i[p]$ fit into short exact sequences,
\[0\rightarrow\mu_p\rightarrow E_i[p]\rightarrow
\Z/p\rightarrow 0.\] 
Suppose that the extension $K(E_1[p],E_2[p])$ is  nontrivial and  unramified over $K$. Then the Galois symbol $s_p$ vanishes, while $(E_1/p\otimes^M E_2/p)(K)\simeq\Z/p$. In particular, the Mackey functor  $E_1/p\otimes^M E_2/p$ is isomorphic to $\overline{U}^0\otimes^M\overline{U}^0$. 
\end{prop}
\begin{proof} First we show that $(E_1/p\otimes^M E_2/p)(K)\simeq\Z/p$. This follows similarly to the proof of  \autoref{ordord}. Without loss of generality, we assume that the extension $K(E_1[p])$ is nontrivial and unramified over $K$. We have an exact sequence of Mackey functors, 
\[0\rightarrow\langle u \rangle\rightarrow\overline{U}^0\stackrel{j}{\longrightarrow} E_1/p\rightarrow[E_1/\hat{E}_1]/p\rightarrow 0,\] where $u$ is the Serre-Tate parameter of $E_1$. Here we denoted by $\langle u \rangle$ the Mackey sub-functor of $\overline{U}^0$ generated by $u$. Note that for a finite extension $L$ of $K$, $\langle u \rangle(L)$ is either $0$ or $\Z/p$ depending on whether $u$ is a $p$th power in $L$ or not.  When we apply $\otimes^M\overline{U}^0$
 to the above sequence, the Mackey functor $\langle u \rangle\otimes^M \overline{U}^0$ vanishes. For, if $L/K$ is a finite extension such that $u\not\in L^{\times p}$, then the extension $L(\sqrt[p]{u})$ is unramified over $L$ and therefore the norm map $\overline{U}^0_L\rightarrow\overline{U}^0_K$ is surjective. We  conclude that in this case the surjection 
$\overline{U}^0\otimes^M\overline{U}^0\stackrel{j}{\longrightarrow} E_1/p\otimes^M E_2/p\rightarrow 0$ is an isomorphism, and hence $(E_1/p\otimes^M E_2/p)(K)\simeq\Z/p$. 

Next we show that $s_p=0$. To show this, we use local Tate duality. The local  Tate duality pairing for the finite $G_K$-module $E_1[p]\otimes E_2[p]$ and for $i=0,1,2$ is a perfect pairing,  
\[\langle\cdot,\cdot\rangle:H^i(K,E_1[p]\otimes E_2[p])\times H^{2-i}(K,\Hom(E_1[p]\otimes E_2[p],\mu_p))\rightarrow\Z/p.\] Using the Weil pairing and the fact that elliptic curves are self dual abelian varieties, the above pairing  for $i=2$ becomes, 
\[\langle\cdot,\cdot\rangle:H^2(K,E_1[p]\otimes E_2[p])\times \Hom_{G_K}(E_1[p],E_2[p])\rightarrow\Z/p.\]
According to the main theorem of \cite[Theorem 1.1]{Gazaki2018}, the orthogonal complement under $\langle\cdot,\cdot\rangle$ of the image of $s_p$ consists precisely of those homomorphisms $f:E_1[p]\rightarrow E_2[p]$ that extend to a homomorphism $\tilde{f}:\mathcal{E}_1[p]\rightarrow\mathcal{E}_2[p]$ of finite flat group schemes over $\Spec(\mathcal{O}_K)$, where $\mathcal{E}_i$ is the N\'{e}ron model of $E_i$ for $i=1,2$. Since both elliptic curves have ordinary reduction, the above subgroup of $\Hom_{G_K}(E_1[p],E_2[p])$ has a simpler description. Namely, according to \cite[Proposition 8.8]{Gazaki2018}, the orthogonal complement of $\img(s_p)$ is the subgroup 
\[H=\{f\in\Hom_{G_K}(E_1[p],E_2[p]): f(E_1[p]^\circ)\subset E_2[p]^\circ\}.\] We will show that every $G_K$-homomorphism $f:E_1[p]\rightarrow E_2[p]$ lies in $H$, which will imply that $s_p=0$. By the assumption of the proposition, $E_i[p]^\circ\simeq\mu_p$, for $i=1,2$, and the $G_K$ action on $E_i[p]$ is upper triangular of the form $\left(\begin{array}{cc}
1 & \alpha_i(\sigma)\\
0 & 1
\end{array}\right)$, for $\sigma\in G_K$, where $\alpha_i:G_K\rightarrow\Hom(\Z/p,\mu_p)$. Since we assumed that the extension $K(E_1[p],E_2[p])$ is nontrivial, at least one of the two characters $\alpha_i$ is nonzero of order exactly $p$. 

Let $f:E_1[p]\rightarrow E_2[p]$ be a $G_K$-homomorphism. After we consider splittings as abelian groups (and not as $G_K$-modules), $E_i[p]\simeq\mu_p\oplus\Z/p$, we can write $f$ in a matrix form $f=\left(\begin{array}{cc}
f_1 & f_2\\
f_3 & f_4
\end{array}\right)$. We want to show that the function $f_3:E_1[p]^\circ\xrightarrow{f} E_2[p]\rightarrow E_2[p]^{et}$ vanishes. This follows by the equality of matrices,
\[\left(\begin{array}{cc}
f_1 & f_2\\
f_3 & f_4
\end{array}\right)\left(\begin{array}{cc}
1 & \alpha_1(\sigma)\\
0 & 1
\end{array}\right)=\left(\begin{array}{cc}
1 & \alpha_2(\sigma)\\
0 & 1
\end{array}\right)\left(\begin{array}{cc}
f_1 & f_2\\
f_3 & f_4
\end{array}\right),\] which yields $\alpha_2(\sigma) f_3=f_3\alpha_1(\sigma)=0$, for every $\sigma\in G_K$. 

\end{proof}

\begin{cor}\label{tower1} \autoref{mainintro2} holds for a product $X=E_1\times E_2$ of elliptic curves with good ordinary reduction. 
\end{cor}
\begin{proof} As usual, let $n\geq 0$ be the largest integer such that $E_i[p^n]\subset E_i(K)$ for $i=1,2$. By extending to a finite extension if necessary, we may assume that $\mu_p\subset K$ and for $i=1,2$ we have short exact sequences of $G_K$-modules, 
\[0\rightarrow\mu_{p^{n+1}}\rightarrow E_{i}[p^{n+1}]\rightarrow\Z/p^{n+1}\rightarrow 0.\] The only case when we need to extend the base field is when the extension \[L_1:=K(E_1[p^{n+1}],E_2[p^{n+1}])\]  is unramified over $K$. After extending to $L_1$ we examine whether the extension $L_2:=K(E_1[p^{n+2}],E_2[p^{n+2}])$ has wild ramification over $L_1$. After repeating this process finitely many times, we get 
	an extension $L_{r+1}/L_r$, for some $r\geq 1$, that has wild ramification. Indeed, 
	there is a largest integer $N>n$ such that $\mu_{p^N}\subset K$, so $L_{r+1}/K$ has wild ramification for some $r\geq 1$. Choosing $r$ the smallest with that property, we have that $L_{r+1}/L_r$ has wild ramification and injectivity holds over $L_r$. 

\end{proof}

\subsection*{Structural Results} 
We next consider the Albanese kernel, $T(E_1\times E_2)$. 
Recall that by the work of Raskind and Spiess (\cite{Raskind/Spiess2000}) we have an isomorphism, \[T(E_1\times E_2)\simeq K(K;E_1,E_2)\simeq T(E_1\times E_2)_{\dv}\oplus (\text{finite}),\] where we denoted by $T(E_1\times E_2)_{\dv}$ the maximal divisible subgroup of the Albanese kernel. \autoref{ordord} and \autoref{kernelord} allow us in most cases to fully determine the finite summand of $T(E_1\times E_2)$. 



\begin{cor}\label{structure1} Let $X=E_1\times E_2$ be the product of two elliptic curves over $K$ with good ordinary reduction. Let $n\geq 0$ be the largest nonnegative integer such that $E_i[p^n]\subset E_i(K)$ for $i=1,2$. Assume that the extension $K(E_1[p^{n+1}],E_2[p^{n+1}])$ has wild ramification.  
Then we have an isomorphism for the Albanese kernel,
\[T(X)\simeq
\left\{
        \begin{array}{ll}
      T(X)_{\dv}\oplus\Z/p^n, & \text{if }n\geq 1 \\
     T(X)_{\dv}, & \text{if } n=0.
        \end{array}
      \right.
\] 
\end{cor}
\begin{proof} \underline{Case 1:} Assume that $n\geq 1$. 
We consider first the special case when for $i=1,2$ we have short exact sequences of $G_K$-modules, \[0\rightarrow \mu_{p^{n+1}}\rightarrow E_i[p^{n+1}]\rightarrow\Z/p^{n+1}\rightarrow 0.\] In this case the corollary follows directly from \autoref{ordord} and \autoref{image1}, since we already know by the computations of Hiranouchi (\autoref{image1}) that \[K(K;E_1, E_2)/p^n\simeq\Z/p^n.\] 

For the general case, let $L/K$ be the smallest finite extension such that the special case holds for the product $E_{1,L}\times E_{2,L}$. It suffices to show that for every $m\geq 1$ the norm map, \[N_{L/K}:K(L;E_{1,L}, E_{2,L})/p^m\rightarrow K(K;E_1,E_2)/p^m\] is surjective. For, the surjectivity of the norm together with \autoref{ordord} will imply that for every $s>n$ the $K$-group $p^s K(K;E_1,E_2)$ is $p$-divisible. At the same time \autoref{image1} gives us an isomorphism, $K(K;E_1,E_2)/p^n\simeq\Z/p^n$, and hence the norm will in fact be an isomorphism, from which the claim follows. 

The surjectivity of the norm follows easily, since $L/K$ is a finite unramified extension. For such extensions, the norm map $N_{L/K}:E_1(L)\rightarrow E_1(K)$ is surjective \cite[Corollary 4.4]{Mazur1972}. Using the projection formula \eqref{projectionformula} of the Somekawa $K$-group, we can easily show that $N_{L/K}:K(L;E_{1,L},E_{2,L})\rightarrow K(K;E_1,E_2)$ is also surjective.  


\underline{Case 2:} Assume that $n=0$. 
Let $L/K$ be the smallest finite extension such that $\mu_p\subset L$ and the $G_L$-modules $E_i[p]$ fit into short exact sequences,
\begin{eqnarray}\label{ses11}0\rightarrow\mu_p\rightarrow E_i[p]\rightarrow\Z/p\rightarrow 0.\end{eqnarray}
The assumption of \autoref{structure1} implies that for at least one $i\in\{1,2\}$ the sequence \eqref{ses11} does not split. More importantly, it corresponds to a non-trivial Serre-Tate parameter $u\in\mathcal{O}^\times_K/\mathcal{O}_K^{\times p}$ which is such that the extension $L(\sqrt[p]{u})/L$ is totally ramified. In this case, the argument is exactly the same as when $n\geq 1$. Namely, \autoref{ordord} gives us that the $K$-group $K(L;E_{1,L}, E_{2,L})$ is $p$-divisible. The general case follows again by the surjectivity of the norm map, \[K(L;E_{1,L}, E_{2,L})/p\xrightarrow{N_{L/K}} K(K;E_1,E_2)/p.\] Note that in this case we have a tower, $K\subset L_0\subset L$ with $L_0/K$ unramified and $L$ is at most $L_0(\mu_p)$. Since the latter is an extension of degree coprime to $p$, the norm map $N_{L/L_0}$ is surjective.

\end{proof}

\subsection*{The case of complex multiplication}
We close the story of two ordinary reduction elliptic curves by considering a very special case, the one of a product $X=E_1\times E_2$ of two elliptic curves both having complex multiplication by an imaginary quadratic field. We note that this case is only partially covered by \autoref{structure1}. Namely, if $E_i[p]\subset E_i(K)$ for $i=1,2$, then we can apply \autoref{ordord} and \autoref{structure1} for $X$. 

The reason we cannot apply \autoref{structure1} when  $E_i[p]\not\subset E_i(K)$ for at least one $i$ is the following.  After extending to a tower, $K\subset L_0\subset L$, such that $L_0/K$ is unramified and $L$ is at most $L_0(\mu_p)$, the $G_L$-module $E_i[p]$ fits into a short exact sequence,
\[0\rightarrow\mu_p\rightarrow E_i[p]\rightarrow\Z/p\rightarrow 0.\] This sequence splits over $L$, that is, the corresponding Serre-Tate parameter is trivial. This follows by \cite[A.2.4]{Serre1998}. 
We therefore have an isomorphism, \[T(E_1\times E_2\times L)\simeq\Z/p\oplus T(E_1\times E_2\times L)_{\dv}.\] 

\begin{que} What happens over $K$? Is $T(X)$ divisible? Or is its finite summand isomorphic to $\Z/p$?
\end{que}
We do not have a method to answer this question in great generality, but we can say something when $K$ is unramified over $\Q_p$. This case is of particular importance, as it will give us global-to-local applications (see \autoref{global-to-local}). 

\begin{prop}\label{CM1} Let $X=E_1\times E_2$ be the product of two elliptic curves over $K$ with good ordinary reduction. Assume that $K$ is unramified over $\Q_p$ and that the elliptic curves $E_1, E_2$ have complex multiplication by an imaginary quadratic field. Then $T(E_1\times E_2)$ is divisible. 
\end{prop}
\begin{proof}
By the usual surjectivity of the norm argument, we may reduce to the case when \[L:=K(E_1[p],E_2[p])=K(\mu_p).\]
By \autoref{ordord}, we have an isomorphism, $K(L;E_{1,L}, E_{2,L})/p\simeq\Z/p$ and this $K$-group can be generated by symbols of the form $\{w,b\}_{L/L}$ with $w\in E_1[p](L)$. We can even choose $w\in\hat{E}_1[p](\mathcal{O}_L)$. We will show that the norm, \[N_{L/K}:K(L;E_{1,L}, E_{2,L})/p\rightarrow K(K;E_1,E_2)/p\] vanishes. Since it is also surjective, this will imply that $K(K;E_1,E_2)$ is $p$-divisible. Since $K(L;E_{1,L}, E_{2,L})/p$ is cyclic, it suffices to show that $N_{L/K}(\{w,b\}_{L/L})=0$ for some nontrivial symbol $\{w,b\}_{L/L}$. Recall from \autoref{formalgroup} that the $p$-torsion  point $w$  satisfies \[w\in\hat{E}_1(\mathfrak{m}^{e_0(L)}_L)\setminus\hat{E}_1(\mathfrak{m}^{e_0(L)+1}_L).\] Since we assumed that $K$ is unramified over $\Q_p$, we have $e_0(L)=1$ and hence $w\in\hat{E}_1(\mathfrak{m}^{1}_L)\setminus\hat{E}_1(\mathfrak{m}^{2}_L)$. By \cite[Theorem 2.1.6]{kawachi2002} we get that the image of $w$ in $\overline{U}_L^0$ lies in $\overline{U}_L^1\setminus\overline{U}_L^2$. This means that the jump of the ramification filtration of $L(\frac{1}{p}w)/L$ is exactly at $p-1$. Coming to the symbol $\{w,b\}_{L/L}$, we can write $b=(b_1, b_2)$ for the image of $b$ in $E_2(L)/p\simeq\overline{U}_L^0\oplus\overline{U}_L^{p}$. We therefore conclude that if $\{w,b\}_{L/L}\neq 0$, then $b_1\in\overline{U}_L^{p-1}$. 

We next consider the restriction map, 
\[\res_{L/K}: \hat{E}_2(\mathcal{O}_K)/[p]\hat{E}_2(\mathcal{O}_K)\rightarrow \hat{E}_2(\mathcal{O}_L)/[p]\hat{E}_2(\mathcal{O}_L).\] Since $e(L/K)=p-1$, we can easily see that the image of $\res_{L/K}$ lies in the subgroup $\hat{E}_2(\mathfrak{m}^{p-1}_L)/([p]\hat{E}_2(\mathcal{O}_L)\cap \mathfrak{m}^{p-1}_L)$. Because $L/K$ is totally ramified, the image in fact equals this subgroup. Using \cite[Theorem 2.1.6]{kawachi2002} once more we conclude that the image of 
\[\hat{E}_2(\mathcal{O}_K)/[p]\hat{E}_2(\mathcal{O}_K)\xrightarrow{\res_{L/K}} \hat{E}_2(\mathcal{O}_L)/[p]\hat{E}_2(\mathcal{O}_L)\xrightarrow{\simeq}\overline{U}_L^0\] is exactly $\overline{U}_L^{p-1}$. Thus, we can find a generator of the $K$-group $K(L;E_{1,L}, E_{2,L})/p$ of the form $\{w,\res_{L/K}(b')\}_{L/L}$ for some $b'\in E_2(K)$. The projection formula \eqref{projectionformula} yields an equality,
\[N_{L/K}(\{w,\res_{L/K}(b')\}_{L/L})=\{N_{L/K}(w),b'\}_{K/K}.\] But the latter  symbol is zero. For, $w$ is a $p$-torsion point of $E_1(L)$. Since the norm is a homomorphism, the same is true for $N_{L/K}(w)$. But the formal group $\hat{E}_1(\mathcal{O}_K)$ is torsion-free, and hence $N_{L/K}(w)=0$.

\end{proof}

\vspace{2pt}

\subsection{The product of two elliptic curves, one with ordinary and the other with supersingular reduction}  

In this subsection, we consider the product of two elliptic curves $E_1$ and $E_2$ over $K$, and we assume $E_1$ has ordinary reduction and $E_2$ has supersingular reduction. 
We recall that, when $E_i[p]\subset E_i(K)$ for $i=1,2$,  we  have an isomorphism
\[s_p: K(K;E_1,E_2)/p\xrightarrow{\simeq}(E_1\otimes^M E_2)/p \xrightarrow{\simeq} \Z/p\oplus \Z/p. \] We proceed similarly to the case of two ordinary reduction elliptic curves, considering first some cases when the injectivity of $s_{p^n}$ can be verified for every $n\geq 1$. 
\begin{prop}\label{ordss} Let $E_1,E_2$ be elliptic curves with good reduction over $K$, where $E_1$ has ordinary reduction, and $E_2$ has supersingular reduction. The Galois symbol $s_{p^n}$ is injective for every $n\geq 1$ in each of the following cases.
\begin{enumerate}
\item When $E_i[p^n]\subset E_i(K)$ for $i=1,2$ and some $n\geq 1$ which is the largest with this property, and there is some $w\in E_1[p^n]$ such that under the decomposition $E_1(K)/p\stackrel{\simeq}{\longrightarrow} \overline{U}_K^0\oplus\overline{U}_K^{pe_0(K)}$, $w$ can be written in the form $w=(w_1,w_2)$ 
	with $w_1\in \overline{U}_K^i\setminus \overline{U}_K^{i+1}$ for some $i$ coprime to $p$ and such that
	\[i \leq \min \{pt(K), p(e_0(K)-t(K))\}.\]
	\item When $E_2[p]\subset E_2(K)$, 
	and the $G_K$-module $E_1[p]$ fits into a short exact sequence of the form \eqref{ses0} having a non-trivial Serre-Tate parameter $u$  such that $u\in\overline{U}_K^i\setminus\overline{U}_K^{i+1}$ for some $i$ coprime to $p$ with $i \leq \min \{pt(K), p(e_0(K)-t(K))\}.$ In this case the $K$-group $K(K;E_1,E_2)$ is $p$-divisible. 
	\end{enumerate}
\end{prop}
\begin{proof} Note that the assumption $E_2[p]\subset E_2(K)$ implies $\mu_p\subset K^\times$. Recall that we have a decomposition as Mackey functors, \[E_2/p\stackrel{\simeq}{\longrightarrow}\overline{U}^{p(e_0(K)-t(K))}\oplus\overline{U}^{pt(K)},\] where $t(K)$ is the invariant of $E_2$ defined in section 2. 
%
%
	
 We first consider case (1). From \autoref{torsionpointscriterion}, it is enough to prove that the $K$-group $K(K;E_1,E_2)/p$ can be generated by symbols of the form $\{a,b\}_{K/K}$ with either $a\in E_1[p^n]$ or $b\in E_2[p^n]$.
	The key idea is to show that there exist elements $b_1\in \overline{U}_K^{p(e_0(K)-t(K))}$ and $b_2 \in  \overline{U}_K^{pt(K)}$ such that $s_p(\{w,(b_1, 0)\}_{K/K}) \neq 0$ and  $s_p(\{w,(0,b_2)\}_{K/K}) \neq 0$, and therefore, by injectivity of $s_p$, elements of the form $\{w,(b_1,0)\}_{K/K}$ and $\{w,(0,b_2)\}_{K/K}$ generate $K(K;E_1,E_2)/p$. 
	
	We have that 
	 $w_1\in \overline{U}_K^i\setminus \overline{U}_K^{i+1}$ for some $i<pe_0(K)$ coprime to $p$ that satisfies 
	\[
	i  \leq \min \{pt(K), p(e_0(K)-t(K))\}.
	\]

	If $i + p e_0(K) - pt(K) \leq pe_0(K)$, then, from the computations in \cite[Lemma 3.4]{Hiranouchi2014}, we can find $b_1 \in \overline{U}_K^{p(e_0(K)-t(K))}$ such that $\{w,(b_1,0)\}_{K/K}\neq 0$; on the other hand, if  $i + pt(K) \leq p e_0(K)$, then we can find $b_2 \in \overline{U}_K^{pt(K)}$ such that $\{w,(0,b_2)\}_{K/K}\neq 0$. The first case occurs if $i \leq pt(K)$, while the second one occurs if $i\leq p(e_0(K)-t(K))$. Since we have $i \leq \min \{pt(K), p(e_0(K)-t(K))\}$, we
	get the result. 
\medskip

We now consider case (2). As usual, using \autoref{formalgroup} we get 
 a surjection of Mackey functors, 
\[\overline{U}^0\otimes^M E_2/p\stackrel{j}{\longrightarrow} (E_1\otimes ^M E_2)/p\rightarrow 0.\] The first Mackey functor is isomorphic to $(\overline{U}^0\otimes^M\overline{U}^{pt(K)})\oplus(\overline{U}^0\otimes^M\overline{U}^{p(e_0(K)-t(K))})$. \cite[Lemma 3.4]{Hiranouchi2014} gives that 
the group $\overline{U}^0\otimes^M\overline{U}^{pt(K)}\oplus\overline{U}^0\otimes^M\overline{U}^{p(e_0(K)-t(K))}(K)$ is isomorphic via the classical Galois symbol to $\Z/p\oplus\Z/p$. Moreover, because we assumed that $u\in\overline{U}_K^i\setminus\overline{U}_K^{i+1}$ with $i$ coprime to $p$ and $i \leq \min \{pt(K), p(e_0(K)-t(K))\}$,  the same lemma \cite[Lemma 3.4]{Hiranouchi2014} together with the discussion preceding it imply that there exist $b_1,b_2\in E_2(K)/p$ such that $\{u,b_1\}_{K/K}$ and $\{u,b_2\}_{K/K}$ generate $(\overline{U}^0\otimes^M E_2/p)(K)$. But both these symbols are mapped to zero under $j$. We conclude that $(E_1\otimes^M E_2)(K)/p=0$ and in particular that the $K$-group $K(K;E_1,E_2)$ is $p$-divisible.

	
\end{proof}
Similarly to the case of two ordinary reduction elliptic curves, we can show that if some of the assumptions of \autoref{ordss} are not satisfied, the Mackey functor Galois symbol has a non-trivial kernel. This is the purpose of the next proposition.
\begin{prop}\label{kernelss} Let $E_1,E_2$ be elliptic curves with good reduction over $K$, where $E_1$ has ordinary reduction, and $E_2$ has supersingular reduction. We assume that $E_2[p]\subset E_2(K)$,  
	and the $G_K$-module $E_1[p]$ fits into a short exact sequence of the form \eqref{ses0} having a non-trivial Serre-Tate parameter $u$. Suppose that $u\in\overline{U}_K^i\setminus\overline{U}_K^{i+1}$ for some integer $i$ such that $i>\min\{pt(K),p(e_0(K)-t(K))\}$. Then the Galois symbol $s_p$ vanishes, while the group $(E_1/p\otimes^M E_2/p)(K)$ contains $\Z/p$.

\end{prop}
\begin{proof} The proof is very analogous to the proof of \autoref{kernelord}, so we give a less detailed analysis of the argument. All the cases to consider are similar, so without loss of generality we assume that $pt(K)<i<p(e_0(K)-t(K))$. In this case we will show that $s_p=0$, while the Mackey functor $E_1/p\otimes^M E_2/p$ is isomorphic to $\overline{U}^0\otimes^M\overline{U}^{p(e_0(K)-t(K))}$. 

We consider the exact sequence of Mackey functors, 
\[\langle u \rangle\otimes^M E_2/p\rightarrow\overline{U}^0\otimes^M E_2/p\xrightarrow{j} E_1/p\otimes^M E_2/p\rightarrow 0,\] where $\langle u \rangle$ is the Mackey functor defined in the proof of \autoref{ordord}.  We have a decomposition $\langle u \rangle\otimes^M E_2/p\simeq(\langle u \rangle\otimes^M\overline{U}^{pt(K)})\oplus(\langle u \rangle\otimes^M\overline{U}^{p(e_0(K)-t(K))})$. 
Similarly to the proof of \autoref{ordss}, because we assumed  $i<p(e_0(K)-t(K))$, we can find $x\in\overline{U}_K^{pt(K)}$ such that the symbol $\{u,x\}_{K/K}\in(\langle u \rangle\otimes^M\overline{U}^{pt(K)})(K)\subset(\overline{U}^0\otimes^M \overline{U}^{pt(K)})(K)$ is non-trivial. On the other hand, the inequality $i>pt(K)$ implies that $\langle u \rangle\otimes^M\overline{U}^{p(e_0(K)-t(K))}=0$. We conclude that $E_1/p\otimes^M E_2/p\simeq \overline{U}^0\otimes^M\overline{U}^{p(e_0(K)-t(K))}$. 

Next we show that $s_p=0$. We consider the orthogonal complement of $\img(s_p)$ under local Tate duality. According to \cite[Proposition 8.11]{Gazaki2018}, the complement in this case coincides with the following subgroup of $\Hom_{G_K}(E_1[p],E_2[p])$, 
\[\{f\in\Hom_{G_K}(E_1[p],E_2[p]): f(E_1[p]^\circ)=0\}.\]
The $G_K$-action on $E_1[p]$ is of the form 
$\left(\begin{array}{cc}
1 & \alpha(\sigma)\\
0 & 1
\end{array}\right)$, for $\sigma\in G_K$, where $\alpha:G_K\rightarrow\Hom(\Z/p,\mu_p)$ is a non-trivial character. Let $f\in\Hom_{G_K}(E_1[p],E_2[p])$. Once again we may write $f$ in a matrix form, 
$f=\left(\begin{array}{cc}
f_1 & f_2\\
f_3 & f_4
\end{array}\right).$ We have a matrix equality,
\[\left(\begin{array}{cc}
f_1 & f_2\\
f_3 & f_4
\end{array}\right)=\left(\begin{array}{cc}
f_1 & f_2\\
f_3 & f_4
\end{array}\right)\left(\begin{array}{cc}
1 & \alpha(\sigma)\\
0 & 1
\end{array}\right),\] which yields $f_1\alpha(\sigma)=f_3\alpha(\sigma)=0$, for every $\sigma\in G_K$ and hence $f_1=f_3=0$. This means exactly that $f$ vanishes when restricted to $E_1[p]^\circ$. 

\end{proof}

Our main goal now is to show that when the assumptions of \autoref{ordss} do not hold, we can construct a tower $K\subset K_1\cdots\subset K_r$ of finite extensions of $K$ so that the weaker criterion described in \autoref{vanishing} holds over $K_r$.  
This will imply that there exists a large enough integer $N\geq 1$ such that the $K$-group  $p^N K(K;E_1,E_2)$ is $p$-divisible. In particular \autoref{maintheoremintro} holds for $E_1\times E_2$.  We start with the following lemma. 
\begin{lem}\label{extendzeta} Let $K$ be a $p$-adic field containing a primitive $p^2$th root of unity, $\zeta_{p^2}$. Let $u\in\overline{U}_K^0$ be such that $u\in\overline{U}_K^i\setminus\overline{U}_K^{i+1}$, where $0<i<pe_0(K)$ and $i$ is coprime to $p$. Let $L=K(\sqrt[p]{u})$. Write $v=\sqrt[p]{u}$. Then $v\in\overline{U}_L^i\setminus\overline{U}_L^{i+1}$.
\end{lem}
\begin{proof}
	We have $v\in  \overline{U}_L^j \setminus \overline{U}_L^{j+1}$ for some $j$. We will show that $j=i$. Write $M= L(\sqrt[p]{v})$, and we therefore have a tower $K\subset L\subset M$ of totally ramified degree $p$ extensions.
	Using Takemoto's computation of the Hasse-Herbrand function (\cite[Lemma 2.2 (2)]{takemoto}), we have  
	\[
	\psi_{L/K}(t) = \begin{cases}
	t, & 0\leq t \leq pe_0(K) - i \\ pt - p e_K + (p -1) i, & pe_0(K)-i\leq t, 
	\end{cases}
	\]  
	\[
	\psi_{M/K}(t) = \begin{cases}
	t, & 0\leq t \leq pe_0(K) - i \\ pt - p e_K + (p -1) i, & pe_0(K)-i\leq t \leq p e_0(K)+e_K -i \\ p^2 t - 2 p^2 e_K + (p^2 -1)i, &  p e_0(K)+e_K -i \leq t
	\end{cases}
	\]  
	and
	\[
	\psi_{M/L}(t) = \begin{cases}
	t, & 0\leq t \leq p^2e_0(K) - j \\  p t - p^2 e_K + (p -1)j, & p^2 e_0(K) - j \leq t
	\end{cases}
	\]  
	
	The function $\psi_{M/K}= \psi_{M/L}\circ\psi_{L/K}$ is non-differentiable at two points in $(0,\infty)$; one corresponds to the unique point where $\psi_{L/K}$ is non-differentiable, $pe_0(K)-i$, and the other corresponds to the unique point where $\psi_{M/L}$ is non-differentiable. On one hand, the second point where $\psi_{M/K}$ is non-differentiable is $\tilde{t}=  p e_0(K)+e_K -i$. On the other hand,  $\psi_{M/L}$ is non-differentiable at $p^2 e_0(K) - j$, so we have
	\[
	\psi_{L/K}(\tilde{t})=p^2 e_0(K) - j.
	\]
	Using Takemoto's formula for $\psi_{L/K}$ and the fact that $\tilde{t}>pe_0(K)-i$, we get
	\begin{align*}
	p^2 e_0(K) - j&=\psi_{L/K}(p e_0(K)+e_K -i)\\&=
	 p(p e_0(K)+e_K -i) - pe_K +(p-1)i \\&=p^2 e_0(K) - i.
	\end{align*}
	Hence we conclude that $i=j$.

\end{proof}
\begin{theo} \label{ordsszeta} Let $E_1,E_2$ be elliptic curves with good reduction over $K$, where $E_1$ has ordinary reduction, and $E_2$ has supersingular reduction. There exists a positive integer $N\geq 1$ such that  the group
	$p^N K(K;E_1,E_2)$ is $p$-divisible. In particular, the Albanese kernel of the product $X=E_1\times E_2$ is the direct sum of a finite group and a divisible group. 
\end{theo}
\begin{proof} 
We start by extending the base field $K$ to a finite extension $K_1$ which is such that,
\begin{itemize}
\item  $E_i[p]\subset E_i(K_1)$, for $i=1,2$,
\item $\mu_{p^2}\subset K_1$,
\item $\hat{E}_1[p^n]\subset \hat{E}_1(\mathcal{O}_{K_1})$ for some $n\geq 1$ and $n$ is the largest with this property. 
\end{itemize}
We consider a $p^n$-torsion point $w_0\in \hat{E}_1[p^n]$ such that $w_0$ does not lie in the image of $[p]:\hat{E}_1(\mathcal{O}_{K_1})\rightarrow\hat{E}_1(\mathcal{O}_{K_1})$. Simply speaking, $\frac{1}{p}w_0\not\in E_1(K_1)$. We consider the decomposition $E_1(K_1)/p\simeq\overline{U}_{K_1}^0\oplus\overline{U}_{K_1}^{pe_0(K_1)}$ and we write $w_0=(w_{0,1}, w_{0,2})$. Assume that $w_{0,1}\in\overline{U}_{K_1}^i\setminus\overline{U}_{K_1}^{i+1}$ for some $0\leq i\leq pe_0(K_1)$. 
Let $t=t(K_1)$ be the $t$-invariant of the  elliptic curve $E_2$ over $K_1$. If $i\leq\min\{pt(K_1),p(e_0(K_1)-t(K_1))\}$, then we can imitate the method of \autoref{ordss} to find elements $b_1,b_2\in E_2(K_1)/p$ such that $\{w_0,b_1\}_{K_1/K_1}$ and $\{w_0,b_2\}_{K_1/K_1}$ generate $K(K_1;E_1,E_2)/p$. Using \autoref{vanishing}, in this case we get that the group $p^n K(K_1;E_1,E_2)$ is $p$-divisible. 

If the index $i$ does not satisfy the above inequality, we perform the following algorithmic process. By extending the base field if necessary, we may assume that the extension $L_1=K_1(\frac{1}{p}w_0)/K_1$ has wild ramification. This in particular means that the index $i$ is coprime to $p$. Note that this will always happen eventually (see \autoref{tower1}). We fix an element $w_1\in \hat{E}_1(\mathcal{O}_{L_1})$ such that $[p]w_1=w_0$. Moreover we write  $w_1=(w_{1,1}, w_{1,2})\in E_1(L_1)/p$. The claim is that $w_{1,1}\in\overline{U}_{L_1}^i\setminus\overline{U}^{i+1}_{L_1}$. This follows by \autoref{extendzeta}. Note that in order to apply this lemma, we needed to assume $\mu_{p^2}\subset K_1$. 

Next, notice that $e_0(L_1)=pe_0(K_1)$ and $t(L_1)=pt(K_1)$. We check again whether \[i\leq\min\{p^2t(K_1),p^2(e_0(K_1)-t(K_1))\}.\] If not, we repeat the process, adding more torsion points of the formal group until we find a finite extension $L_{r}$ such that $i\leq\min\{p^{r+1}t(K_1),p^{r+1}(e_0(K_1)-t(K_1))\}$. We conclude that for some $r\geq 0$ the group 
$p^{n+r} K(L_r;E_1,E_2)$ is $p$-divisible. 

To finish the argument, we need to show that \autoref{maintheoremintro} holds for the product $X=E_1\times E_2$. 
 Let $s=[L_r:K]=p^l\cdot m$, where $m$ is coprime to $p$. Let $x\in K(K;E_1,E_2)$. We have, 
$s\cdot x=N_{L_r/K}(\res_{L_r/K}(x))$. Set $y=m\cdot x$ and $v=\res_{L_r/K}(x)\in K(L_r;E_1,E_2)$, so that we have an equality
$p^l y=N_{L_r/K}(v)$. By the previous step, we get 
\[p^{l+n+r}y=N_{L_r/K}(p^{n+r}v)\equiv 0 \text{ mod }p^{l+n+r+1} N_{L_r/K}(K(L_r;E_1,E_2)).\] Since the $K$-group $K(K;E_1,E_2)$ is $m$-divisible (\cite[Theorem 3.5]{Raskind/Spiess2000}), the above relation holds for every $y\in K(K;E_1,E_2)$.  We can therefore set $N=l+n+r$ to make the statement of the theorem true.

\end{proof}

\vspace{2pt}

\subsection{When one curve is a Tate elliptic curve} In this section we extend our computations to the case when at least one of the two curves is a Tate elliptic curve. When both $E_1, E_2$ are Tate curves, the injectivity of $K(K;E_1,E_2)/n\rightarrow H^2(K,E_1[n]\otimes E_2[n])$ has been proved by Yamazaki (\cite{Yamazaki2005}), for every $n\geq 1$. We therefore assume that $E_1$  is a Tate elliptic curve and $E_2$  has good reduction. 

We want to proceed as in the previous subsections, giving sufficient (and for the Mackey functor necessary) criteria for the map 
\[K(K;E_1,E_2)/p^n\xrightarrow{s_{p^n}} H^2(K,E_1[p^n]\otimes E_2[p^n])\] to be injective. In this case it suffices to consider only the injectivity of $s_p$ when $\mu_p\subset K$ and the $q$-invariant of $E_1$ is not a $p$th power. Indeed, if $q=q'^{p^n}$ for some $q'\in\mathfrak{m}_K$ that is not a $p$th power and some $n\geq 1$, we can replace $E_1$ with the isogenous elliptic curve $E_1'=\G_m/q'^\Z$. Namely, the $p^n$-power map gives an isogeny, $E_1'\xrightarrow{\phi} E_1$, which induces a map on $K$-groups, 
\[K(K;E_1',E_1)\xrightarrow{\phi} K(K;E_1,E_2). \]
The image of $\phi$ is exactly the subgroup $p^nK(K;E_1,E_2)$. We conclude that proving that the group $p^nK(K;E_1,E_2)$ is $p$-divisible is equivalent to proving that the $K$-group $K(K;E_1',E_2)$ is $p$-divisible.

We start by fixing a uniformizer $\pi_K$ of $K$. The analogue of \autoref{ordord} and \autoref{ordss} in this case is given by the following proposition.  
\begin{prop}\label{Tate1}  Let $E_1$, $E_2$ be elliptic curves over $K$ such that $E_1$ is a Tate curve and $E_2$ has good reduction. Assume $\mu_p\subset K$ and $\hat{E}_2[p]\subset E_2(K)$. Moreover, suppose that  $E_1$ has invariant $q\in\mathfrak{m}_K$ such that $\sqrt[p]{q}\not\in K$. 
\begin{enumerate}
\item If $E_2$ has ordinary reduction, then a necessary and sufficient condition for the injectivity of the Mackey functor Galois symbol 
\[(E_1\otimes^M E_2)(K)/p\xrightarrow{s_p} H^2(K,E_1[p]\otimes E_2[p])\] is that $q\in\mathfrak{m}_K^i\setminus\mathfrak{m}_K^{i+1}$ for some $i$ coprime to $p$. 
\item If $E_2$ has supersingular reduction with $E_2[p]\subset E_2(K)$ and $E_2$ has invariant $t(K)$, then a necessary and sufficient condition for the injectivity of the Mackey functor Galois symbol $s_p$ is that $q$ can be written as $q=\pi_K^iu$, with either $i$ coprime to $p$, or $u\in \overline{U}_K^j\setminus\overline{U}_K^{j+1}$ with $j\leq\min\{pt(K),p(e_0(K)-t(K))\}$. 
\end{enumerate} In the above cases, the $K$-group $K(K;E_1,E_2)$
 is $p$-divisible. 
\end{prop}
\begin{proof} The two cases are very similar, therefore we only prove the proposition when $E_2$ has ordinary reduction. 
 We first prove that, if $q\in\mathfrak{m}_K^i\setminus\mathfrak{m}_K^{i+1}$ for some $i$ coprime to $p$, then $K(K;E_1,E_2)$ is $p$-divisible. 

We have an exact sequence of Mackey functors $\G_m/p\xrightarrow{\varepsilon} E_1/p\rightarrow 0$. Since the functor   $\otimes^M E_2/p$ is right exact,  we get a surjection, 
$(\G_m/p\otimes^M E_2/p)(K)\xrightarrow{\varepsilon} (E_1/p\otimes^M E_2/p)(K)$. We want to show that the group $(E_1/p\otimes^M E_2/p)(K)$ vanishes. 
 Since we assumed that $\sqrt[p]{q}\not\in K$, the above map $\varepsilon$ has a kernel generated by the image of $q$. We can proceed similarly to the proof of the second part of  \autoref{ordord}. Namely, we will show that the group $(\G_m/p\otimes^M E_2/p)(K)$ is generated by symbols of the form $\{q,x\}_{K/K}$ with $x\in E_2(K)$. Since we assumed that $\hat{E}_2[p]\subset E(K)$, we have an exact sequence of Mackey functors,
\[\overline{U}^0\rightarrow E_2/p\rightarrow [E_2/\hat{E}_2]/p\rightarrow 0,\] where $[E_2/\hat{E}_2]$ is the Mackey functor defined in \autoref{ordord}. Since $\G_m/p\otimes^M$ is right exact, we obtain an exact sequence 
\begin{eqnarray}\label{ordtate1}\G_m/p\otimes^M\overline{U}^0\rightarrow \G_m/p\otimes^M E_2/p\rightarrow \G_m/p\otimes^M[E_2/\hat{E}_2]/p\rightarrow 0.\end{eqnarray} It suffices therefore to show that the groups $(\G_m/p\otimes^M\overline{U}^0)(K)$ and $(\G_m/p\otimes^M[E_2/\hat{E}_2]/p)(K)$ can be generated by symbols of the form $\{q,x\}_{K/K}$ with $x\in \overline{U}^0_K$ and $x\in[E_2/\hat{E}_2]/p(K)$, respectively. 

By \cite[Lemma 3.3]{Hiranouchi2014}, $(\G_m/p\otimes^M\overline{U}^0)(K)$ is isomorphic to $\Z/p$ via the classical Galois symbol $g_p$. Since we assumed that $q\in\mathfrak{m}_K^i\setminus\mathfrak{m}_K^{i+1}$ for some $i$ coprime to $p$, we can write $q=\pi_K^i v$, for some unit $v\in U_K^0$.  We consider the extension $L=K(\sqrt[p]{\pi_K})$. According to \cite[Lemma 2.2 (1)]{takemoto}, this extension is totally ramified of degree $p$ whose Galois group $\Gal(L/K)$ has a jump at the ramification filtration at $pe_0(K)$. By \cite[p. 86, Corollary 7]{serre1979local},  there exists $x\in U^{pe_0(K)}_K$ such that $g_p(\{\pi_K,x\}_{K/K})\neq 0$. Therefore, $\{\pi_K,x\}_{K/K}$ generates $(\G_m/p\otimes^M\overline{U}^{0})(K)$. We claim  that $\{q,x\}_{K/K}$ also generates $(\G_m/p\otimes^M\overline{U}^0)(K)$. Indeed, observe that $\{q,x\}_{K/K}=i\{\pi_K,x\}_{K/K}+\{v,x\}_{K/K}$, and, since $K(\sqrt[p]{x})/K$ is unramified,    $\{v,x\}_{K/K}=0$. Since $i$ is coprime to $p$ and $\{\pi_K,x\}_{K/K}\neq 0$, we get $\{q,x\}_{K/K}\neq 0$. It follows that $\{q,x\}_{K/K}$ generates $(\G_m/p\otimes^M\overline{U}^0)(K)$.

The computation for $(\G_m/p\otimes^M[E_2/\hat{E}_2]/p)(K)$ is similar, so we omit it.

\medskip
We next want to show that the condition on the invariant $q$ is necessary for the injectivity of $s_p$ at the level of the Mackey product, $(E_1/p\otimes^M E_2/p)(K)$. We assume that $q=\pi_K^{p^s}v$, for some $s\geq 1$ and some unit $v\in \mathcal{O}_K^\times$ that is not a $p$th power. In this case we claim that the group $(E_1\otimes^M E_2)/p$ contains $\Z/p$, while $s_p=0$. 
The proof is very similar to the proofs of \autoref{kernelord} and \autoref{kernelss}, so we only sketch the argument here. 

The first claim follows by the sequence \eqref{ordtate1}. Namely, we can show that the  map \[(\langle q\rangle\otimes^M E_2/p)(K)\rightarrow (\G_m/p\otimes^M[E_2/\hat{E}_2]/p)(K)\] vanishes. On the other hand, the group $(\G_m/p\otimes^M\overline{U}^0)(K)$ can be generated by a symbol of the form $\{q,x\}_{K/K}$, unless the unit $v$ is in $\overline{U}_K^{pe_0(K)}$. 

To show that the map $s_p$ vanishes, we need to compute again the orthogonal complement of $\img(s_p)$ under Tate duality. We claim that this complement coincides with the following subgroup of $\Hom_{G_K}(E_1[p],E_2[p])$,
$\{f\in\Hom_{G_K}(E_1[p],E_2[p]):f(\mu_p)=0\}.$ The claim follows by the following commutative diagram,
\[\begin{tikzcd}
(\G_m/p\otimes^M E_2/p)(K)\ar{r}{s_p}\ar{d}{\varepsilon} & H^2(K, \mu_p\otimes E_2[p])\ar{d}\\
(E_1/p\otimes^M E_2/p)(K)\ar{r}{s_p} & H^2(K, E_1[p]\otimes E_2[p])
\end{tikzcd} \]
We already saw that the map $\varepsilon$ is surjective. Note that the top $s_p$ is also surjective. This follows by \cite[Theorem 2.9]{Bloch1981}. The rest of the argument is the same as in \autoref{kernelss}. 
 
\end{proof}
We end this section with the analogue of \autoref{ordsszeta}. 
\begin{cor} Let $E_1$ be a Tate curve and $E_2$ a supersingular reduction elliptic curve.  Then the Albanese kernel of the product $X=E_1\times E_2$ is the direct sum of a finite group and a divisible group. 
\end{cor}
\begin{proof}  By \cite[Theorem 3.5]{Raskind/Spiess2000} we have an isomorphism, $T(E_1\times E_2)\simeq D\oplus F$, where $D$ is a group that is $m$-divisible for every integer $m$ coprime to $p$ and $F$ is a finite group. Let $M$ be the order of $F$. It suffices to show that the group $M\cdot T(E_1\times E_2)$ is the direct sum of a $p$-divisible group and a finite group. To do this we imitate the proof of \autoref{ordsszeta}, which applies almost verbatim. In fact it becomes even easier, since we can construct a tower of finite extensions, $L_r\supset \cdots\supset L_1\supset K$, by attaching roots of unity, $\{\zeta_{p^n}:n_0\leq n\leq n_0+r\}$, which are considered as torsion points of the Tate curve $E_1$. \autoref{extendzeta} reassures that after attaching a finite number of roots of unity the assumptions of \autoref{Tate1} will eventually hold over $L_r$. The rest of the proof is exactly the same as \autoref{ordsszeta} but now applied to the $m$-divisible group $D$.

\end{proof}

\vspace{1pt}

\subsection{The product of more than two curves} 
To finish the proof of theorems \eqref{maintheoremintro} and \eqref{mainintro2}, we need to consider also the case of the $K$-group $K(K;E_1,\cdots, E_r)$ attached to more than two elliptic curves. Everything will follow from the next corollary, which  generalizes previous computations of Raskind and Spiess (\cite[Theorem 4.5, Remark 4.4.5]{Raskind/Spiess2000}) and Hiranouchi (\cite[Proposition 4.3]{Hiranouchi2014}).
\begin{cor}\label{morecurves}
 Let $E_1,\cdots, E_r$ be elliptic curves over $K$ with split semistable reduction. Assume that $E_1$ does not have supersingular reduction and if $E_i$ has supersingular reduction for some $i\geq 2$, then $E_i[p]\subset E_i(K)$.  If $r\geq 3$, the $K$-group $K(K;E_1,\cdots, E_r)$ is $p$-divisible. 
\end{cor}
\begin{proof} 
We will show that the Mackey product, $(E_1/p\otimes^M\cdots\otimes^M E_r/p)(K)=0$. Using associativity of the product, it suffices to prove the claim when $r=3$. Moreover, using the surjectivity of the norm, we may assume that $\mu_p\subset K$ and if $E_i$ has ordinary reduction for some $i\in\{2,3\}$, then $\hat{E}_i[p]\subset E_i(K)$. 

We will prove the corollary in the following two specific cases to illustrate the method. Any other case can be proved in a very analogous way.
\begin{itemize}
\item Assume that all three curves have good ordinary reduction. 
Imitating the proof of \autoref{ordord}, we can prove a surjection of Mackey functors,  
\[\overline{U}^0\otimes^M\overline{U}^0\otimes^M\overline{U}^0\rightarrow (E_1\otimes^M E_2\otimes^M E_3)/p\rightarrow 0.\] The claim then follows after we observe that the functor $\overline{U}^0\otimes^M\overline{U}^0\otimes^M\overline{U}^0$ vanishes by \cite[Lemma 4.2.1]{Raskind/Spiess2000}. 
\item Assume that $E_1$ has ordinary reduction and $E_2, E_3$ have supersingular reduction. Since we assumed that $E_i[p]\subset E_i(K)$ for $i=2,3$, we can use \autoref{decomposition} to compute the Mackey functors $E_i/p$. We consider the product $(E_1\otimes^M E_2)/p$. By \autoref{ordss} and \autoref{kernelss} we get that this product is either 0 or isomorphic to a direct sum of Mackey functors of the form $\overline{U}^0\otimes^M\overline{U}^s$ for some $s>0$. By \cite[Lemma 3.3]{Hiranouchi2014} we get an isomorphism,
\[\overline{U}^0\otimes^M\overline{U}^s\simeq(\G_m\otimes^M\G_m)/p\simeq \overline{U}^0\otimes^M\overline{U}^0.\] Therefore the Mackey functor, $(E_1\otimes^M E_2\otimes^M E_3)/p$ can be decomposed in direct pieces that look like $\overline{U}^0\otimes^M\overline{U}^0\otimes^M E_3/p$. By imitating the argument for the product $\overline{U}^0\otimes^M E_3/p$, the claim follows once again by \cite[Lemma 4.2.1]{Raskind/Spiess2000}. 
\end{itemize}
	

\end{proof}
\vspace{1pt}

\vspace{1pt}
\subsection{The Brauer-Manin pairing} We close this section by discussing a corollary about the Brauer-Manin pairing, 
\[CH_0(X)\times Br(X)\rightarrow\Q/\Z\] between the group $CH_0(X)$ and the Brauer group, $Br(X)$ of $X$.  This corollary follows directly by the work of Yamazaki (\cite[Theorem 1.2]{Yamazaki2005}). 
\begin{cor}\label{Brauer-Manin} Let $E_1,\cdots,E_r$ be elliptic curves over a $p$-adic field $K$ with good reduction and such that at least one does not have supersingular reduction. Let $X=E_1\times\cdots\times E_r$. There is a finite extension $L$ of $K$ such that the left kernel of the Brauer Manin pairing 
\[CH_0(X\times_K L)\times Br(X\times_K L)\rightarrow\Q/\Z\] is the maximal divisible subgroup of $CH_0(X\times_K L)$. 
\end{cor}
\begin{proof} Yamazaki showed the left kernel of the pairing \[CH_0(X)\times Br(X)\rightarrow\Q/\Z\] is the maximal divisible subgroup of $CH_0(X)$ if and only if the cycle map $c_{p^n}$ is injective for every $n\geq 1$. The corollary then follows by \autoref{mainintro2}. 

\end{proof}
\section{Applications over number fields}\label{global-to-local} We close this paper by suggesting a conjecture for varieties defined over algebraic number fields. In this section, $K$ shall denote a number field. We will denote by $\Omega$ the set of all places of $K$.

\begin{conj}\label{almostall} Let $X$ be  smooth projective geometrically connected variety over a number field $K$ such that $X$ has a $K$-rational point. Let $X_v:=X\times_K K_v$ be the base change to the completion, $K_v$, at a finite place $v$ of $K$. If the Albanese kernel, $T(X_v)$, is the direct sum of a divisible group, $D_v$, and a finite group, $F_v$, then $T(X_v)$ is divisible for almost all finite places $v$ of $K$, that is $F_v=0$ for almost all places $v$.  
\end{conj} 

We already saw that even the local picture has only been  established in very few cases. However, we saw that for $X=E_1\times\cdots\times E_d$ a product of elliptic curves over a $p$-adic field $K_v$, \autoref{CTconj} has been established in most cases. We will focus on this case to provide two pieces of evidence towards \autoref{almostall}. 
\subsection*{Complex multiplication}
The first evidence comes from \autoref{ordord}. In fact, the following is a direct corollary of \autoref{CM1}
\begin{cor}\label{CMcase} Let $X=E\times E$ be the self-product of an elliptic curve over an algebraic number field $K$. Assume that $E$ has complex multiplication by an imaginary quadratic field $M$. Then the Albanese kernel,  $T(X_v)$, is divisible for almost all ordinary reduction places $v$ of $K$. 
\end{cor}
\begin{proof} Let $v$ be an ordinary reduction place of $X$. Assume that $v$ lies above a rational prime $p$.
The corollary follows immediately from \autoref{CM1} after we observe that for all but finitely many $v$ the extension $K_v/\Q_p$ is unramified.

\end{proof}
\subsection*{Local-to-global expectations} Our second motivation for \autoref{almostall} comes from a local-to-global conjecture  for zero-cycles and its compatibility with a famous conjecture of Beilinson (\cite{Beilinson1984}) and Bloch (\cite{Bloch1984}). The latter predict that for a  smooth projective geometrically connected variety $X$ over a number field $K$, the Albanese kernel $T(X)$ is a finite group.  On the other hand, the Brauer-Manin pairing gives rise to a complex 
\[\widehat{A_0(X)}\stackrel{\Delta}{\longrightarrow}
\widehat{A_{0,\mathbf{A}}(X)}\rightarrow\Hom(Br(X)/Br(K),\Q/\Z),\] where for an abelian group $M$ we denote by $\widehat{M}:=\varprojlim_{n} M/nM$. 
The \textit{adelic Chow group}, $A_{0,\mathbf{A}}(X)$ is defined to be $\prod_{v\text{ finite}}A_0(X_v)$ when $K$ is a totally imaginary number field. In the general case it has a more complicated expression. It was  originally proposed by Colliot-Th\'{e}l\`{e}ne and Sansuc (\cite[Section 4]{Colliot-Thelene/Sansuc1981}) that for geometrically rational varieties  the above complex is exact. This conjecture was later generalized to arbitrary varieties by
\color{black}
 Kato and Saito  (\cite[7.6.2]{Kato/Saito1983}, see also \cite[p. 394]{Saito1989}). For more information on the local-to-global conjecture we refer to \cite[Conj\'{e}ctures 1.5]{Colliot-Thelene1995} and \cite[Section 2.6]{Wittenberg2018}. 

From now on suppose $X=E_1\times E_2$ is the product of two elliptic curves over $K$. Restricting the above complex to the Albanese kernel gives rise to a complex, 
\[\widehat{K(K;E_1,E_2)}\stackrel{\Delta}{\longrightarrow}
\widehat{K(K;E_{1v}, E_{2v})}\rightarrow\Hom(Br(X)/Br_1(X),\Q/\Z),\] where $Br_1(X):=\ker(Br(X)\rightarrow Br(X_{\overline{K}}))$. By a result of Skorobogatov and Zarhin (\cite{Skorogogatov/Zharin2014}) the quotient $Br(X)/Br_1(X)$ is finite. Assuming \autoref{CTconj} is true for $X$, it implies that the middle term of the complex is an infinite product of finite groups. The only way that all three conjectures are compatible with each other is if in the group $\widehat{K(K;E_{1v}, E_{2v})}$ only finitely many places give nontrivial contribution. That is, the group $K(K;E_{1v}, E_{2v})$ should be divisible for all but finitely many places $v$ of $K$.

\vspace{3pt}

\vspace{10pt}

\bibliographystyle{amsalpha}

\bibliography{bibfile}

\end{document}